\documentclass[a4paper]{article}
\usepackage[latin1]{inputenc}
\usepackage{anysize}
\usepackage{graphicx}

\usepackage{amssymb,amsmath}
\usepackage{fancyhdr}

\usepackage{amsthm}
\theoremstyle{plain}
\usepackage{tikz}
\usepackage{enumerate,enumitem}
\usepackage{thmtools}

\newcommand{\cc}{\mathcal}
\newcommand{\bb}{\mathbb}
\newcommand{\ra}[1]{\overrightarrow{#1}}
\newcommand{\la}[1]{\overleftarrow{#1}}
\newcommand{\lam}[1]{\lambda\left( {#1} \right)}

\declaretheorem{theorem}
\numberwithin{theorem}{section}
\newtheorem{lemma}[theorem]{Lemma}  
\newtheorem{cor}[theorem]{Corollary}

\newtheorem{defn}[theorem]{Definition}
\newtheorem{claim}[theorem]{Claim}
\newtheorem*{remark}{Remark}

\newcommand{\N}{\mathbb{N}}

\MakeRobust{\overrightarrow}

\author{Joshua Erde\thanks{The author was supported by the Alexander von Humboldt Foundation.} }
\title{A unified treatment of linked and lean tree-decompositions}   
\date{}

\begin{document}
\maketitle

\begin{abstract}
There are many results asserting the existence of tree-decompositions of minimal width which still represent local connectivity properties of the underlying graph, perhaps the best-known being Thomas' theorem that proves for every graph $G$ the existence of a linked tree-decompositon of width tw$(G)$. We prove a general theorem on the existence of linked and lean tree-decompositions, providing a unifying proof of many known results in the field, as well as implying some new results. In particular we prove that every matroid $M$ admits a lean tree-decomposition of width tw$(M)$, generalizing the result of Thomas.
\end{abstract}

\section{Introduction}
Given a graph $G$ a \emph{tree-decomposition} is a pair $(T,\cc{V})$ consisting of a tree $T$, together with a collection of subsets of vertices $\cc{V} = \{ V_t \subseteq V(G) \, : \, t \in V(T)\}$, called \emph{bags}, such that: 
\begin{itemize}
\item $V(G) = \bigcup_{t \in T} V_t$;
\item For every edge $e \in E(G)$ there is a $t$ such that $e$ lies in $V_t$;
\item $V_{t_1} \cap V_{t_3} \subseteq V_{t_2}$ whenever $t_2 \in t_1 T t_3$.
\end{itemize}

The \emph{width} of this tree-decomposition is the number $\max \{ |V_t| -1 \, : \, t \in V(T) \}$ and its \emph{adhesion} is $\max \{ |V_t \cap V_{t'}| \, :\, (t,t') \in E(T) \}$. Given a graph $G$ its \emph{tree-width} tw$(G)$ is the smallest $k$ such that $G$ has a tree-decomposition of width $k$. 

\begin{defn}\label{d:link}
A tree decomposition $(T,\cc{V})$ is called \emph{linked} if for all $k \in \bb{N}$ and every $t,t' \in T$, either $G$ contains $k$ disjoint $V_{t}$-$V_{t'}$ paths or there is an $s \in tTt'$ such that $|V_s|<k$.
\end{defn}
Robertson and Seymour showed the following:

\begin{theorem}[Robertson and Seymour \cite{RS86}]\label{t:badlink}
Every graph admits a linked tree-decomposition of width $<3.2^{\text{tw}(G)}$.
\end{theorem}
This result was an essential part of their proof that for any $k$ the set of graphs of tree-width less than $k$ is well quasi ordered by the minor relation. Thomas gave a new proof of Theorem \ref{t:badlink}, improving the bound on the tree-width of the linked tree-decomposition from $3.2^{\text{tw}(G)} - 1$ to the best possible value of tw$(G)$.

\begin{theorem}[Thomas \cite{T90}]\label{t:thomlink}
Every graph $G$ admits a linked tree-decomposition of width $\text{tw}(G)$.
\end{theorem}
In fact he showed a stronger result.

\begin{defn}\label{d:leangraph}
A tree decomposition $(T,\cc{V})$ is called \emph{lean} if for all $k \in \bb{N}$, $t,t' \in T$ and vertex sets $Z_1 \subseteq V_{t}$ and $Z_2 \subseteq V_{t'}$ with $|Z_1|=|Z_2|=k$, either $G$ contains $k$ disjoint $Z_1-Z_2$ paths or there exists an edge $(s,s') \in tTt'$ with $|V_s \cap V_{s'}| < k$.
\end{defn}

Thomas showed that every graph has a lean tree-decomposition of width tw$(G)$. It is relatively simple to make a lean tree-decomposition linked without increasing its width: we subdivide each edge $(s,s') \in T$ by a new vertex $u$ and add $V_u := V_{s} \cap V_{s'}$. The real strength of Definition \ref{d:leangraph} in comparison to Definition \ref{d:link} is the case $t=t'$. Broadly, Definition \ref{d:link} tells us that the 'branches' in the tree-decomposition are no larger than the connectivity of the graph requires, if the separators $V_{s} \cap V_{s'}$ along a path in $T$ are large, then $G$ is highly connected along this branch. The case $t=t'$ of Definition \ref{d:leangraph} tells us that the bags are also no larger than their 'external connectivity' in $G$ requires.

Bellenbaum and Diestel \cite{BD02} used some of the ideas from Thomas' paper to give short proofs of two known results concerning tree-decompositions, the first Theorem \ref{t:thomlink} and the other the tree-width duality theorem of Seymour and Thomas \cite{ST93}. Very similar ideas appear in the literature in multiple proofs of the existence of `lean' tree-decompositions of minimal width, for many different width parameters. For example $\theta$-tree-width \cite{CDHH14,GJ16}, directed path-width \cite{KS15,K14} and DAG-width \cite{K14}. Similar ideas also appear in the proof of Geelen, Gerards and Whittle \cite{GGW02} that every matroid $M$ admits a `linked' branch-decomposition of width the branch width of $M$, which was used to extend Robertson and Seymour's result about the well quasi-ordering of graphs of bounded tree-width to matroids (See also \cite{O05} for an application to rank-width).

In this paper we will prove generalizations of Theorem \ref{t:thomlink} in a general framework introduced by Diestel \cite{O14}, which give unifying proofs of the existence of `linked' or `lean' tree-decompositions for a broad variety of width parameters. In particular this theorem will imply all the known results for undirected graphs and matroids from the introduction, as well as many new results by applying it to other width parameters expressible in this framework. In particular we prove a generalization of Theorem \ref{t:thomlink} to matroids\footnote{See Section \ref{s:mat} for the relevant definitions.}:

{
\renewcommand{\thetheorem}{\ref{t:matroid}}
\begin{theorem}
Every matroid $M$ admits a lean tree-decomposition of width tw$(M)$.
\end{theorem}
\addtocounter{theorem}{-1}
}

In order to prove a result broad enough to cover both graph and matroid tree-width, it will be necessary to combine some of the ideas from the proof of Bellenbaum and Diestel \cite{BD02} with that of Geelen, Gerards and Whittle \cite{GGW02}. In particular, we note that, when interpreted in terms of tree-decompositions of graphs, our proof will give a slightly different proof of Theorem \ref{t:thomlink} than appears in \cite{BD02}.

In the next section we will introduce the necessary background material. In Sections \ref{s:link} and \ref{s:lean} respectively we will give proofs of our theorems on the existence of linked and lean tree-decompositions. Finally, in Section \ref{s:app} we will use these theorems to deduce results about different width parameters of graphs and matroids.

\section{Background material}
\subsection{Notation}
Given a tree $T$ and vertices $t_1,t_2 \in V(T)$ we will write $t_1Tt_2$ for the unique path in $T$ between $t_1$ and $t_2$. Similarly, given edges $e_1,e_2 \in E(T)$ then we will write $e_1Te_2$ for the unique path in $T$ which starts at $e_1$ and ends at $e_2$. If $\ra{e}$ is a directed edge in $T$, say $\ra{e} = (x,y)$, then we will write $\la{e} = (y,x)$, and we denote by $T(\ra{e})$ the subtree of $T$ formed by taking the component $C$ of $T-x$ which contains $y$, and adding the edge $(x,y)$. For general graph theoretic notation we will follow \cite{D10}.

\subsection{Separation systems}
Our objects of study will be the separation systems of Diestel and Oum. A more detailed introduction to these structures can be found in \cite{O14}.

A \emph{separation system} $(\ra{S},\leq,*)$  as defined in \cite{O14} consists of a partially ordered set $\ra{S}$, whose elements are called \emph{oriented separations}, together with an order reversing involution $*$. One particular example of a separation system is the following: Given a set $V$, the set of ordered pairs
\[
\text{sep}(V) = \{ (A,B) \, : \, A \cup B = V \}
\]
together with the ordering
\[
(A,B) \leq (C,D) \text{ if and only if } A \subseteq C \text{ and } B \supseteq D,
\]
and the involution $(A,B)^* = (B,A)$ forms a separation system. Furthermore, any subset of sep$(V)$ which is closed under involutions forms a separation system, and we call such a separation system a \emph{set separation system}. We will use $V$ to denote the ground set of a set separation system unless otherwise specified. We will work with set separation systems rather than in full generality for a variety of reasons. Some of the results in the paper will carry over to the more abstract setting, but only with so many technical restrictions that little, if anything, is lost by this restriction. 

The elements $(A,B)$ of $\ra{S}$ are called \emph{separations} and the corresponding \emph{separation} is $\{A,B\}$. $(A,B)$ and $(B,A)$ are then the \emph{orientations} of $\{A,B\}$\footnote{When the context is clear we will often refer to both oriented and unoriented separations as merely `separations' to improve the flow of the text.}.

Two separations $\{A,B\}$ and $\{C,D\}$ are \emph{nested} if they have $\leq$-comparable orientations. A set of separations is \emph{nested} if it's elements are pairwise nested\footnote{We will often use terms defined for separations in reference to it's orientations, and vice versa.}. A (multi-)set of separations $\sigma \subset \ra{S}$ is a \emph{star} if
\[
(A,B) \leq (D,C) \text{ for all submultisets } \{(A,B),(C,D)\} \subset \sigma.
\]

There are binary operations on sep$(V)$ such that $(A,B) \wedge (C,D)$ is the infimum and $(A,B) \vee (C,D)$ is the supremum of $(A,B)$ and $(C,D)$ in sep$(V)$ given by:
\[
(A,B) \wedge (C,D) = (A \cap C, B \cup D) \text{    and    }   (A,B) \vee (C,D) = (A \cup C, B \cap D).
\]
If a set separation system $\ra{S}$ is closed under $\wedge$ and $\vee$ we say it is a \emph{universe} of set separations. The following lemma, whose elementary proof we omit, is often useful.
\begin{lemma}\label{l:corner}
If $(A,B)$ is nested with $(C,D)$ and $(E,F)$ then it is also nested with $(C,D) \vee (E,F), (C,D)  \wedge (E,F), (C,D)  \vee (F,E)$ and $(C,D) \wedge (F,E)$.
\end{lemma}

We call a function $(A,B) \mapsto |A,B|$ on a universe of set separations an \emph{order function} if it is non-negative, symmetric and submodular. That is, if $0 \leq |A,B| = |B,A|$ for all $(A,B) \in \ra{S}$ and
\[
|A \cap C, B \cup D| + |A \cup C, B \cap D| \leq |A,B|+ |C,D|
\]
for all $(A,B),(C,D) \in \ra{S}$. Given a universe of set separations $\ra{S}$ with any order function we will often write ${\ra{S}\!}_k$ for the separation system
\[
{\ra{S}\!}_k := \{ (A,B) \in \ra{S} \,:\, |A,B| < k \}.
\] 

If $r: 2^V \rightarrow \N$ is a non-negative submodular function then it is easy to verify that
\[
|X,Y|_r = r(X) + r(Y) - r(V)
\]
is an order function on any universe contained in sep$(V)$. We note the following lemma, which will be useful later.

\begin{lemma}\label{l:un}
Let $\ra{S}$ be a universe of set separations with an order function $|.|_r$ for some non-negative non-decreasing submodular function $r : 2^V \rightarrow \N$. If $(B,A) \leq (C,D)$, then $|A \cup C, B \cap D|_r \leq |C,D|_r$.
\end{lemma}\vspace{-10pt}
\begin{proof}
Note that $A \cup C = V$ and $B  \cap D \subset C \cap D$. Hence, 
\begin{align*}
|A \cup C, B \cap D|_r &= r(A \cup C) + r(B \cap D) - r(V) \\
&= r(B \cap D) \leq r(C \cap D) \\ 
&\leq r(C) + r(D) - r(C \cup D) = |C,D|_r.
\end{align*}
\end{proof}

Our main application of separation systems will be to separations of graphs and matroids. If $G$ is a graph, we say that an ordered pair $(A,B)$ of subsets of $V(G)$ is an oriented separation if $A\cup B = V(G)$ and there is no edge between $A \setminus B$ and $B \setminus A$. It is not hard to show that the set of oriented separations $\ra{S}$ of a graph $G$ forms a universe of set separations as a subset of sep$\left(V(G)\right)$, and if we consider the non-negative non-decreasing submodular function $r:2^{V(G)} \rightarrow \N$ given by $r(A) := |A|$, then the order function given by $r$ is $|X,Y|_r = |X| + |Y| - |V| =  |X \cap Y|$. Throughout this paper, whenever we consider separation systems of graph separations we will use this order function.

Similarly, if $M = (E,\cc{I})$ is a matroid with rank function $r$, we say that an ordered bipartition $(A,E\setminus A)$ of the edge set is an oriented separation. Note that $r: 2^{E} \rightarrow \N$ is a non-negative non-decreasing submodular function. Again, the set of oriented separations $\ra{S}$ of a matroid $M$ forms a universe of set separations, with an order function given by $|X,Y|_r = r(X) + r(Y) - r(E)$. Again, whenever we consider separations systems of matroid separations we will use this order function.

\subsection{$S$-trees over sets of stars}

Given a tree $T$ and $t \in V(T)$ we write
\[
{\ra{F}\!}_t := \{ \ra{e} \,:\, \ra{e} = (x,t) \in \ra{E}(T) \}
\]

\begin{defn}
Let $\ra{S}$ be a set separation system. An \emph{$S$-tree} is a pair $(T,\alpha)$ where $T$ is a finite tree and 
\[
\alpha : \ra{E}(T) \rightarrow \ra{S}
\]
is a map from the set of oriented edges of $T$ to $\ra{S}$ such that, for each edge $e \in E(T)$, if $\alpha(\ra{e}) = (A,B)$ then $\alpha(\la{e}) = (B,A)$. 
\end{defn}

Let us call an $S$-tree \emph{tame} if each $\sigma_t$ is a star. It is easy to check that if $(T,\alpha)$ is a tame $S$-tree then $\alpha$ preserves the natural ordering on $\ra{E}(T)$. We note that, in contrast to [\cite{O15}, Section 6], we are considering $\sigma_t$ as a multiset of separations. In \cite{O15}, since they only consider the underlying sets, it can happen that a separation $(A,B)$ appears twice in $\sigma_t$, even if $(A,B) \not\leq (B,A)$, which will make $\alpha$ not order-preserving. In \cite{O15} this problem is avoided by 'pruning' the $S_k$-trees, by removing branches behind duplicated separations, so as only to work with trees in which each $\sigma_t$ is a set. In order to reduce technical details in the proofs of Theorem \ref{t:link} and \ref{t:lean} we will instead work with multisets. For a finite set $X$ we will write $\N^{X}$ for the set of all finite submultisets of $X$.

If $\cc{F} \subseteq \N^{\ra{S}}$ is a family of stars we say that an $S$-tree $(T,\alpha)$ is \emph{over $\cc{F}$} if for all $t \in V(T)$, the multiset $\sigma_t := \alpha(\ra{F}_t) \in \cc{F}$. Sometimes, for notational convenience, we will refer to an $S$-tree over $\cc{F}$ when $\cc{F} \not\subseteq \N^{\ra{S}}$, by which it should be taken to mean an $S$-tree over $\cc{F} \cap \N^{\ra{S}}$.  Note that if $\cc{F}$ is a family of stars, then ever $S$-tree over $\cc{F}$ is tame.

It is observed in \cite{DO142} that many existing ways of decomposing graphs can be expressed in this framework if we take $\ra{S}$ to be the universe of separations of a graph. For example given a star $\sigma = \{ (A_1,B_1), (A_2,B_2), \ldots , (A_n,B_n)\}$ let us write $\langle \sigma \rangle := | \bigcap_i B_i |$. If we let
\[
\cc{F}_k = \{ \sigma \in \N^{\ra{S}} \, : \, \sigma \text{ a star, } \langle \sigma \rangle <k\}
\]
then it is shown in \cite{DO142} that a graph $G$ has a tree-decomposition of width $< k-1$ if and only if there exists an $S_k$-tree over $\cc{F}_k$. More examples will be given later in Section \ref{s:app} when we apply Theorems \ref{t:link} and \ref{t:lean} to existing types of tree-decomposition.

Given two separations $(A,B) \leq (C,D)$ in a universe of set seperations $\ra{S}$ with an order function $|.|$ let us define 
\[
\lam{(A,B),(C,D)} =\min \{ |X,Y| \, : \, (A,B) \leq (X,Y) \leq (C,D) \}.
\]
We will think of $\lambda$ as being a measure of connectivity, and indeed in the case of separation systems of graphs and matroids it will coincide with the normal connectivity function. 

\begin{defn}\label{d:linksep}
Let $\ra{S}$ be a universe of set separations with an order function $|.|$, $k \in \N$. We say that an $S_k$-tree $(T,\alpha)$  over some family of stars is \emph{linked} if for every pair of edges $\ra{e} \leq \ra{f}$ in $\ra{E}(T)$ the following holds:
\[
\min \{ |\alpha(\ra{g})| \,:\, \ra{e}\leq \ra{g}\leq \ra{f} \} = \lam{ \alpha(\ra{e}), \alpha(\ra{f}) }.
\]
\end{defn}

This definition more closely resembles the definition of linked from Geelen, Gerards and Whittle \cite{GGW02} than Definition \ref{d:link}. However, as in the introduction, if we take a tree-decomposition of a graph which is linked in this sense and add the adhesion sets as parts, it will be linked in the sense of Thomas. For any family of stars $\cc{F}$ we also introduce the following concept of leanness.  

\begin{defn}
Let $\ra{S}$ be a set separation system and let $(T,\alpha)$ be an $S$-tree over some family of stars $\cc{F}$. Given a vertex $t \in T$ we say that a separation $(A,B)$ is \emph{addable at $t$} if $\sigma_t \cup \{(A,B)\} \in \cc{F}$.
\end{defn}

\begin{defn}\label{d:leansep}
Let $\ra{S}$ be a universe of set separations with an order function $|.|_r$ for some non-negative non-decreasing submodular function $r: 2^V \rightarrow \N$, $k \in \N$. We say that an $S_k$-tree $(T,\alpha)$ over a family of stars $\cc{F}$ is \emph{$\cc{F}$-lean} if for every pair of vertices $t$ and $t'$ and every pair of separations $(A,B) \leq (B',A')$ such that $(A,B)$ is addable at $t$ and $(A',B')$ is addable at $t'$, either
\[
\lam{ (A,B) , (B',A') } \geq  \min \{ r(A), r(A') \}
\]
or
\[
\min \{ |\alpha(\ra{g})| \,:\, \ra{g} \in tT t' \} = \lam{ (A,B) , (B',A')}.
\]
\end{defn}

We note that, unlike in the case of Definitions \ref{d:link} and \ref{d:leangraph}, it is not true in general that the existence of an $\cc{F}$-lean $S_k$-tree over $\cc{F}$ will imply the existence of a linked $S_k$-tree over $\cc{F}$. Indeed, for many families $\cc{F}$ of stars, the set of separations addable at any star may be very limited, if not empty. In this case the property of being $\cc{F}$-lean may be vacuously true of any $S_k$-tree over $\cc{F}$. 

\subsection{Shifting $S$-trees}
Let $\ra{S}$ be a universe of set separations. Our main tool will be a method for taking an $S$-tree $(T,\alpha)$ and a separation $(X,Y)$ and building another $S$-tree $(T',\alpha')$  which contains $(X,Y)$ as the image of some edge. 

Given a tame $S$-tree $(T,\alpha)$, a separation $(A,B) = \alpha(\ra{e}) \in \alpha(\ra{E}(T))$, and another separation $(A,B)\leq (X,Y)$ the \emph{shift of $(T,\alpha)$ onto $(X,Y)$ (with respect to $\ra{e}$)} is the $S$-tree $(T(\ra{e}),\alpha')$ where $\alpha'$ is defined as follows:
 
For every edge $f \in E(T(\ra{e}))$ there is a unique orientation of $f$ such that $\ra{e} \leq \ra{f}$. We define
\[
\alpha'(\ra{f}) := \alpha(\ra{f}) \vee (X,Y) \text{    and    } \alpha'(\la{f}) := \alpha'(\ra{f})^* = \alpha(\la{f}) \wedge (Y,X).
\]
Note that, since $(A,B) \leq (X,Y)$ we have that $\alpha'(\ra{e}) = \alpha(\ra{e}) \vee (X,Y) = (A,B) \vee (X,Y) = (X,Y)$, and so $(X,Y) \in \alpha(\ra{E}(T(\ra{e})))$.

Our tool will be in essence [Lemma 4.2, \cite{DO141}], however, for technical reasons we will take a slightly different statement. The first reason is that we need a slightly more general notion of shifting, since in \cite{DO141} they only shift a tree with respect to a leaf. Secondly, in \cite{DO141} they define $\alpha'$ as a composition of $\alpha$ with a map from $\ra{S}$ to $\ra{S}$, which causes problems if the same separation appears twice in a tree. We avoid this issue by defining $\alpha'$ in a slightly different way, which allows us to avoid some of the technical restrictions on [Lemma 4.2, \cite{DO141}]. As with the change from sets to multisets in the definition of an $S$-tree, this is only a superficial change in order to avoid messy technical details in the proofs of Theorems \ref{t:link} and \ref{t:lean}.

\begin{lemma}\label{l:tree}
Let $S$ be a universe of set separations. Given a tame $S$-tree $(T,\alpha)$, $\ra{e}$ and $(X,Y)$ as above, the shift of $(T,\alpha)$ onto $(X,Y)$ with respect to $\ra{e}$ is a tame $S$-tree.
\end{lemma}
\begin{proof}
The shift is clearly an $S$-tree by definition of $\alpha'$. So, it remains to check that it is tame. Let $\ra{e} = (\ell,t)$. Since $\ell$ is a leaf, $\alpha'({\ra{F}\!}_\ell)$ is a singleton, and so clearly a star. 

Given $t \in V(T(\ra{e})) \setminus \ell$ there is a unique edge $\ra{f} \in {\ra{F}\!}_t$ such that $\ra{e} \leq \ra{f}$. Then,
\[
\alpha'({\ra{F}\!}_t) = \{ \alpha(\ra{f}) \vee (X,Y) \} \cup \{ \alpha(\ra{g}) \wedge (Y,X) \,: \, \ra{g} \in {\ra{F}\!}_t \setminus \ra{f} \}.
\]
It follows, since $\alpha(\ra{g}) \leq \alpha(\la{h})$ for all $\ra{g},\ra{h} \in {\ra{F}\!}_t \setminus \ra{f}$, that
\[
\alpha'(\ra{g}) = \alpha(\ra{g}) \wedge (Y,X) \leq  \alpha(\ra{g}) \leq \alpha(\la{h}) \leq  \alpha(\la{h}) \vee (X,Y) = \alpha'(\la{h}),
\]
and
\begin{align*}
\alpha'(\ra{f}) = \alpha(\ra{f}) \vee (X,Y) \leq \alpha(\la{g}) \vee (X,Y) = \alpha'(\la{g}).
\end{align*}
Since the involution is order-reversing, it follows that $\alpha'({\ra{F}\!}_t)$ is a star.
\end{proof}

In general, if $\ra{S}$ is not a universe but just a set separation system, it may not be true that the shift of an $S$-tree is still an $S$-tree. However, when $\ra{S}={\ra{S}\!}_k$ for some universe there is a natural condition on $(X,Y)$ which guarantees that the shift will still be an $S_k$-tree.

\begin{defn}
Let $\ra{S}$ be a set separation system living in some universe with an order function $|.|$ and let $(A,B),(X,Y) \in \ra{S}$. We say that $(X,Y)$ is \emph{linked} to $(A,B)$ if $(A,B) \leq (X,Y)$ and
\[
 |X,Y| = \lam{ (A,B) , (X,Y)}.
\]
\end{defn}

\begin{lemma}\label{l:shift1}
Let $\ra{S}$ be a universe of set separations with an order function $|.|$, $k \in \N$, $(T,\alpha)$ a tame $S_k$-tree, $(A,B)=\alpha(\ra{e}) \in \alpha(\ra{E}(T))$ and let $(X,Y)$ be linked to $(A,B)$. Then the shift of $(T,\alpha)$ onto $(X,Y)$ with respect to $\ra{e}$ is a tame $S_k$-tree.
\end{lemma}
\begin{proof}
By Lemma \ref{l:tree} the shift is a tame $S$-tree, so it will be sufficient to show that $\alpha'(\ra{E}(T(\ra{e}))) \subseteq {\ra{S}\!}_k$. Suppose that $(C,D)  = \alpha'(\ra{f}) \in \alpha'(\ra{E}(T(\ra{e})))$. By symmetry we may assume that $\ra{e} \leq \ra{f}$, and so $(C,D) = \alpha(\ra{f}) \vee (X,Y)$ by definition. 

Since $\ra{e}\leq \ra{f}$, $\alpha(\ra{e}) = (A,B) \leq \alpha(\ra{f})$ and also $(A,B) \leq (X,Y)$. Hence,
\[
 (A,B) \leq \alpha(\ra{f}) \wedge (X,Y) \leq (X,Y).
\]
Therefore, since $(X,Y)$ is linked to $(A,B)$, it follows that
\[
|\alpha(\ra{f}) \wedge (X,Y)| \geq |X,Y|
\]
and so by the submodularity of the order function
\[
|C,D| = |\alpha(\ra{f}) \vee (X,Y)| \leq |\alpha(\ra{f})| < k,
\]
and so $(C,D) \in {\ra{S}\!}_k$.
\end{proof}

Finally we will need a condition that guarantees that the shift of an $S_k$-tree over $\cc{F}$ is still over $\cc{F}$. 

\begin{defn}
Let $\ra{S}$ be a set separation system living in some universe with an order function $|.|$ and let $\cc{F} \subseteq \N^{\ra{S}}$. We say that $\cc{F}$ is \emph{fixed under shifting} if whenever $\sigma \in \cc{F}$, $(T,\alpha)$ is a tame $S$-tree with $\ra{e} = (t,t') \in \ra{E}(T)$,  $\sigma = \alpha({\ra{F}\!}_s)$ for some $s \in T(\ra{e}) \setminus t$, and $(X,Y)$ is linked to $\alpha(\ra{e})$, then in the shift of $(T,\alpha)$ onto $(X,Y)$ with respect to $\ra{e}$, $\sigma'=\alpha'({\ra{F}\!}_s) \in \cc{F}$.
\end{defn}

It essentially this property of a family of stars $\cc{F}$ that is used in \cite{DO142} to prove a duality theorem for $S_k$-trees over $\cc{F}$. It may seem like a strong property to hold, but in fact it is seen to hold in a large number of cases corresponding to known width-parameters of graphs such as branch-width, tree-width (see Lemma \ref{l:shift2}), and path-width. 

\begin{lemma}\label{l:closedundershifting}
Let $\ra{S}$ be a universe of set separations with an order function $|.|$, $k \in \mathbb{N}$ and let $\cc{F} \subset \N^{\ra{S}}$ be a family of stars which is fixed under shifting. Let $(T,\alpha)$ be a tame $S_k$-tree, $\ra{e} = (t,t') \in \ra{E}(T)$ such that:
\[
\text{for every } s \in V(T(\ra{e})) \setminus t, \, \alpha\left({\ra{F}\!}_s\right) \in \cc{F}.
\]
If $(X,Y)$ is linked to $\alpha(\ra{e})$ then the shift of $(T,\alpha)$ onto $(X,Y)$ with respect to $\ra{e}$ is a tame $S_k$-tree over $\cc{F} \cup \{(Y,X)\}$. Furthermore, if $\{(Y,X)\} \not\in \cc{F}$ then $t$ is the unique vertex such that $\alpha'\left({\ra{F}\!}_t \right) = \{(Y,X)\}$.
\end{lemma}
\begin{proof}
By Lemma \ref{l:shift1} the shift is a tame $S_k$-tree, it remains to show that it is over $\cc{F} \cup \{(Y,X)\}$. Firstly, we note that by definition $\alpha'(\ra{e}) = (X,Y)$ and so $\alpha'({\ra{F}\!}_t) = \{ \alpha'(\la{e}) \} = \{ (Y,X) \} \in \cc{F} \cup \{(Y,X)\}$.

Suppose then that $s \in V(T(\ra{e})) \setminus t$. By assumption, $\alpha\left({\ra{F}\!}_s\right) \in \cc{F}$ and so, since $\cc{F}$ is fixed under shifting, $\sigma'=\alpha'({\ra{F}\!}_t) \in \cc{F}$.
\end{proof}

Let $\ra{S}$ be a universe of set separations with an order function $|.|_r$ for some non-decreasing submodular function $r: 2^V \rightarrow \mathbb{N}$. For any star 
\[
\sigma = \{ (A_0,B_0), (A_2,B_2), \ldots , (A_n,B_n)\}
\]
let us define the \emph{size} of $\sigma$
\[
\langle \sigma \rangle_r = \sum_{i=0}^n r(B_i) - n . r(V).
\]
Note that, when $r(X) = |X|$, $\langle \sigma \rangle_r = |\bigcap_{i=0}^n B_i|$. We define
\[
\cc{F}_p = \{ \sigma \subset \N^{\ra{S}} \, : \, \sigma \text{ a star, } \langle \sigma \rangle_r < p\}.
\]
We note the following properties of $\langle . \rangle_r$.

\begin{lemma}\label{l:sepstar}
Let $\ra{S}$ be a universe of set separations with an order function $|.|_r$ for some non-negative non-decreasing submodular function $r: 2^V \rightarrow \N$. If $\sigma$ is a star and $(A,B) \in \sigma$ then $\langle \sigma \rangle_r \geq |A,B|_r$. Also, if $\sigma$ is a star and $\sigma \cup \{(A,B) \}$ is a star then $\langle \sigma \rangle_r  \geq r(A)$.
\end{lemma}
\begin{proof}
Both proofs will follow from the following claim:
\begin{claim}\label{c:starint}
Let $Z_0,Z_1, \ldots Z_n, X \subseteq V$ and let us write $Z^*_i = \bigcap_{j=0}^i Z_j $.If $Z^*_i \cup Z_{i+1} = V$ for each $i$, then
\[
\sum_{i=0}^n r(Z_i \cap X)  \geq r(Z^*_n \cap X) + n .r(X).
\]
\end{claim}
\begin{proof}[Proof of Claim]
Since $Z^*_i \cup Z_{i+1} = V$, by submodularity
\[
r(Z^*_i \cap X) + r(Z_{i+1} \cap X)  \geq r(Z^*_{i+1} \cap X ) + r(X).
\]
Since $Z^*_0 = Z_0$, if we add these inequalities for $i=0,\ldots,n-1$, we get
\[
\sum_{i=0}^n r(Z_i \cap X) \geq r(\bigcap_{i=0}^n Z_i \cap X) + n .r(X).
\]
\end{proof}

For the first result let us write $\sigma = \{ (A,B), (A_1,B_1), \ldots , (A_n,B_n)\}$. Since $\sigma$ is a star, $(A,B) \leq (B_i,A_i)$ for all $i$ and so $A \subset \bigcap_{i=1}^n B_i$. Hence $r(A) \leq r(\bigcap_{i=1}^n B_i)$. Therefore, applying Claim \ref{c:starint} to $B_1,B_2,\ldots,B_n,V$, we see that
\begin{align*}
\langle \sigma \rangle_r &= r(B) + \sum_{i=1}^n r(B_i) - n.r(V) \\
&\geq r(B) + r(\bigcap_{i=1}^n B_i) - r(V)\\
&\geq r(B) + r(A) - r(V)\\
&= |A,B|_r
\end{align*}

For the second let us write $\sigma = \{(A_0,B_0), (A_1,B_1), \ldots (A_n,B_n)\}$. If $\sigma \cup \{(A,B) \}$ is a star then $(A,B) \leq (B_i,A_i)$ and so $A \subset B_i$ for all $i$. Hence, by Claim \ref{c:starint} applied to $B_0,B_1,\ldots, B_n,V$,
\begin{align*}
\langle \sigma \rangle_r  &= \sum_{i=0}^n r(B_i) - n.r(V) \\
&\geq r(\bigcap_{i=0}^n B_i)\\
&\geq r(A).
\end{align*}

\end{proof}

The following lemma can be seen as an analogue of [Lemma 2, \cite{BD02}], in that it gives a condition for when the shifting operation does not increase `width', when interpreted as $\max \{ \langle \sigma_t \rangle_r \,:\, t \in V(T) \}$. The proof of this lemma follows closely the proofs of [Lemmas 6.1 and 8.3, \cite{DO142}].

\begin{lemma}\label{l:shift2}
Let $\ra{S}$ be a universe of set separations with an order function $|.|_r$ for some non-negative non-decreasing submodular function $r: 2^V \rightarrow \N$ and let $p \in \mathbb{N}$. Then $\cc{F}_p$ is fixed under shifting
\end{lemma}
\begin{proof}
Suppose that $\sigma \in \cc{F}_p$, $(T,\alpha)$ is a tame $S$-tree with $\ra{e} = (t,t') \in \ra{E}(T)$,  $\sigma = \alpha({\ra{F}\!}_s)$ for some $s \in T(\ra{e}) \setminus t$, and $(X,Y)$ is linked to $\alpha(\ra{e})$. Let $(T',\alpha')$ be the shift of $(T,\alpha)$ onto $(X,Y)$ with respect to $\ra{e}$ and write $\sigma'=\alpha'({\ra{F}\!}_s) \in \cc{F}$.

Let $\sigma = \{(A_0,B_0),(A_1,B_1), \ldots, (A_n,B_n) \}$ and let $\alpha(\ra{e}) = (A,B)$. There is a unique edge $\ra{g} \in {\ra{F}\!}_s$ such that $\ra{e} \leq \ra{g}$, let us suppose without loss of generality that $\alpha(\ra{g}) = (A_0,B_0)$. Then, $(A,B) \leq (A_0,B_0)$. By definition of $\alpha'$,
\[
\sigma' = \alpha'({\ra{F}\!}_s) = \{(A_0 \cup X,B_0 \cap Y),(A_1 \cap Y ,B_1 \cup X), \ldots, (A_n \cap Y,B_n \cup X) \}.
\]
So, it will be sufficient to show that
\[
\langle \sigma' \rangle_r = r(B_0 \cap Y) + \sum_{i=1}^n r(B_i \cup X) - n. r(V) \leq \langle \sigma \rangle_r.
\]
that is,
\begin{equation}\label{e:show}
r(B_0 \cap Y) + \sum_{i=1}^n r(B_i \cup X) - n. r(V) \leq \sum_{i=0}^n r(B_i) - n. r(V) 
\end{equation}

Since $r$ is submodular, it follows that
\[
r(B_0 \cap Y) + r(B_0 \cup Y) \leq r(B_0) + r(Y)
\]
and, for each $i=1,\ldots,n$,
\[
r(B_i \cap X) + r(B_i \cup X) \leq r(B_i) + r(X).
\]
Therefore, in order to show (\ref{e:show}) it will be sufficient to show that
\begin{equation}\label{e:show2}
r(B_0 \cup Y) +  \sum_{i=1}^n r(B_i \cap X) \geq r(Y) + n. r(X)
\end{equation}

By Claim \ref{c:starint} applied to $B_1,B_2,\ldots, B_n,X$,
\[
\sum_{i=1}^n r(B_i \cap X) \geq r\left(\bigcap_{i=1}^n B_n \cap X\right) + (n-1).r(X).
\]

Let us write $A^* = \bigcup_{j=1}^n A_j$ and $B^* = \bigcap_{j=1}^n B_j$. Note that $A \subseteq A_0 \subseteq B^*$ and $B \supseteq B_0 \supseteq A^*$. Since $(X,Y)$ and $(B^*, A^*) \in \ra{S}$, so is $(X \cap B^*,Y \cup A^*)$ and $(A,B) \leq (X \cap B^*,Y \cup A^*) \leq (X,Y)$. So, since $(X,Y)$ is linked to $(A,B)$, it follows that $|X \cap B^*,Y \cup A^*|_r \geq |X,Y|_r$. Therefore by definition of $|.|_r$,
\[
r(B^* \cap X) + r(A^* \cup Y) \geq r(X) + r(Y).
\]

Since $A^* \subset B_0$ and $r$ is non-decreasing, it follows that $r(B_0 \cup Y) \geq r(A^* \cup Y)$ and so we can conclude that
\begin{align*}
r(B_0 \cup Y) +  \sum_{i=1}^n r(B_i \cap X) &\geq  r(B_0 \cup Y) +  r(B^* \cap X) + (n-1).r(X) \\
&\geq r(A^* \cup Y) + r(B^* \cap X) + (n-1).r(X) \\
&\geq r(Y) +  n .r(X).
\end{align*}
\end{proof}

\section{Linked $S_k$-trees}\label{s:link}
In this section we will prove a general theorem on the existence of linked tree-decompositions. The proof follows closely the proof of Geelen, Gerards and Whittle \cite{GGW02}, extending their result to a broader class of tree-decompositions. They consider branch decompositions of integer valued symmetric submodular functions. Given such a function $\lambda$ on a set $V$ we can consider $\lambda$ as an order function on the universe $\ra{S} = \{ (A, V \setminus A)\}$ of bipartitions of $V$ by taking $|A,V\setminus A| = \lambda(A)$, where by scaling by an additive factor we may assume $\lambda$ is a non-negative function. They defined a notion of `linked' for these decompositions and showed the following:

\begin{theorem}\label{t:branch}[\cite{GGW02}, Theorem 2.1] 
An integer-valued symmetric submodular function with branch-width $n$ has a linked branch-decomposition of width $n$.
\end{theorem}

A direct application of Theorem \ref{t:branch} gives analogues of Theorem \ref{t:thomlink} for branch decompositions of matroids or graphs \cite{GGW02}, and also rank-decompositions of graphs \cite{O05}.

It can be shown that a branch-decomposition of $\lambda$ is equivalent to an $S$-tree over a set $\cc{T}$ of stars which is fixed under shifting (See the proof of [\cite{DO142}, Lemma 4.3]), and that the width of this branch-decomposition is the smallest $k$ such that the $S$-tree is an $S_k$-tree. Furthermore, in this way the definition of linked given in $\cite{GGW02}$ coincides with Definition \ref{d:linksep}. 

Let us say that a universe of set separations $\ra{S}$ with an order function $|.|$ is \emph{grounded} if it satisfies the conclusion of Lemma \ref{l:un}, that is, if for every pair of separations if $(B,A) \leq (C,D)$ then $|A \cup C, B \cap D| \leq |C,D|$. So, Lemma \ref{l:un} says that $\ra{S}$ is grounded whenever $|.|=|.|_r$ for some non-negative non-decreasing submodular function $r: 2^{V} \rightarrow \N$. We note that, if $\ra{S}$ is the universe of bipartitions of a set then $\ra{S}$ is grounded for any order function $|.|$. Indeed, since $|.|$ is symmetric and submodular, it follows that for any bipartition $(X,Y)$, 
\[
2|X,Y| = |X,Y| + |Y,X| \geq |X \cup Y , X \cap Y| + |X \cap Y, X \cup Y| = 2|V,\emptyset|
\]
and so, if $(B,A) \leq (C,D)$ then $|A \cup C, B \cap D| = |V, \emptyset | \leq |C,D|$. In this way the following theorem implies, and generalises Theorem \ref{t:branch}.
\begin{theorem}\label{t:link}
Let $\ra{S}$ be a grounded universe of set separations with an order function $|.|$, $k \in \N$, and let $\cc{F} \subset \N^{{\ra{S}\!}_k}$ be a family of stars which is fixed under shifting. If there exists an $S_k$-tree over $\cc{F}$, then there exists a linked $S_k$-tree over $\cc{F}$.
\end{theorem}
\begin{proof}
Let $(T,\alpha)$ be an $S_k$-tree over $\cc{F}$. Note that, since $\cc{F}$ is a family of stars, $(T,\alpha)$ is tame. We write $T_p$ for the subforest of $T$ where $\ra{e} \in T_p$ if and only if $|\alpha(\ra{e})| \geq p$. Let us write $e(T_p)$ for the number of edges of $T_p$ and $c(T_p)$ for the number of components of $T_p$.

We define an order on the set 
\[
\cc{T} = \{ (T,\alpha) \,: \, (T,\alpha) \text{ an } S_k \text{-tree over } \cc{F}\}
\]
as follows. We say that $(T,\alpha) \prec (S,\alpha')$ if there exists a $p \in \mathbb{N}$ such that for all $p'>p$, $e(T_{p'}) = e(S_{p'})$ and $c(T_{p'}) = c(S_{p'})$ and either:
\begin{itemize}
\item $e(T_p) < e(S_p)$, or
\item $e(T_p)=e(S_p)$ and $c(T_p) > c(S_p)$.
\end{itemize}
Let $(T,\alpha)$ be a $\prec$-minimal element of $\cc{T}$. We claim that $(T,\alpha)$ is linked.

Suppose not, that is, there are two edges $\ra{e} \leq \ra{f}$ such that there is no $\ra{g} \in \ra{E}(T)$ with $\ra{e} \leq \ra{g} \leq \ra{f}$ and
\[
|\alpha(\ra{g})| = \lam{\alpha(\ra{e}),\alpha(\ra{f})} =: \lam{(A,B),(D,C)}.
\]
Now, there is some separation $(A,B) \leq (X,Y) \leq (D,C)$ such that 
\[
|X,Y| = \lam{(A,B),(D,C)} =: \ell.
\]
Let us choose such an $(X,Y)$ which is nested with a maximal number of separations in $\alpha\left( \ra{E}(T) \right)$. Note that $(X,Y)$ is linked to $(A,B)$ and $(Y,X)$ is linked to $(C,D)$.

\begin{figure}[ht!]
\begin{center}
\begin{tikzpicture}

\node[draw, circle,scale=.2, fill] at (0,0) {};
\node[draw, circle,scale=.2, fill] at (-3,0) {};
\node[draw, circle,scale=.2, fill] at (-4,1) {};
\node[draw, circle,scale=.2, fill] at (-4,-1) {};
\node[draw, circle,scale=.2, fill] at (1,1) {};
\node[draw, circle,scale=.2, fill] at (1,-1) {};
\node[draw, circle,scale=.2, fill] at (3,0) {};
\node[draw, circle,scale=.2, fill] at (2,1) {};
\node[draw, circle,scale=.2, fill] at (6,0) {};
\node[draw, circle,scale=.2, fill] at (7,1) {};
\node[draw, circle,scale=.2, fill] at (7,-1) {};

\draw (-4,1)--(-3,0)--(0,0)--(3,0)--(6,0)--(7,1) (7,-1)--(6,0) (2,1)--(3,0) (1,1)--(0,0)--(1,-1) (-4,-1)--(-3,0);

\node at (-1.5,0.5) {$\ra{e}$};
\node at (4.5,0.5) {$\ra{f}$};

\end{tikzpicture}
\end{center}
\caption{The tree $T$.}
\end{figure}
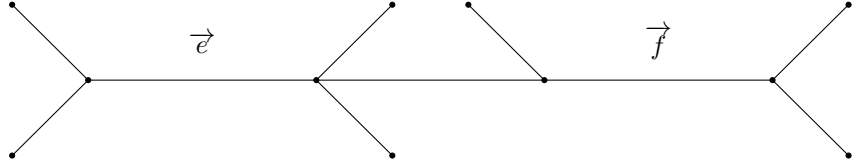

Let $(T_1,\alpha_1)$ be the $S_k$-tree given by shifting $(T,\alpha)$ onto $(X,Y)$ with respect to $\ra{e}$ and let $(T_2,\alpha_2)$ be the $S_k$-tree given by shifting $(T,\alpha)$ onto $(Y,X)$ with respect to $\la{f}$. By Lemmas \ref{l:shift1} and \ref{l:closedundershifting}, $(T_1,\alpha_1)$ and $(T_2,\alpha_2)$ are $S_k$-trees over $\cc{F} \cup \{(Y,X)\}$ and $\cc{F} \cup \{(X,Y)\}$ respectively. For each vertex and edge in $T$ let us write $v_1$ and $v_2$ or $e_1$ and $e_2$ for the copy of $v$ or $e$ in $T_1$ and $T_2$ respectively, if it exists (Note that, since $T_1=T(\ra{e})$ and $T_2 = T(\la{f})$, not every vertex or edge will appear in both trees). We let $(\hat{T},\hat{\alpha})$ be the following $S_k$-tree:

$\hat{T}$ is the tree formed by taking the disjoint union of $T_1$ and $T_2$, and identifying the edge $\ra{e}_1$ with the edge $\ra{f}_2$. $\hat{\alpha}$ is then formed by taking union of $\alpha_1$ and $\alpha_2$ on the domain $\ra{E}(\hat{T})$ where we note that, since $\alpha_1(\ra{e}) = (X,Y) = \alpha_2(\ra{f})$ this map is well defined. By Lemma \ref{l:closedundershifting} $(\hat{T},\hat{\alpha})$ is an $S_k$-tree over $\cc{F}$.

\begin{figure}[ht!]
\begin{center}
\begin{tikzpicture}[scale=0.6]

\node[draw, circle,scale=.2, fill] at (0,0) {};
\node[draw, circle,scale=.2, fill] at (-3,0) {};
\node[draw, circle,scale=.2, fill] at (-6,0) {};
\node[draw, circle,scale=.2, fill] at (3,0) {};
\node[draw, circle,scale=.2, fill] at (6,0) {};
\node[draw, circle,scale=.2, fill] at (9,0) {};
\node[draw, circle,scale=.2, fill] at (-7,1) {};
\node[draw, circle,scale=.2, fill] at (-2,1) {};
\node[draw, circle,scale=.2, fill] at (-1,1) {};
\node[draw, circle,scale=.2, fill] at (4,1) {};
\node[draw, circle,scale=.2, fill] at (5,1) {};
\node[draw, circle,scale=.2, fill] at (10,1) {};
\node[draw, circle,scale=.2, fill] at (-7,-1) {};
\node[draw, circle,scale=.2, fill] at (-2,-1) {};
\node[draw, circle,scale=.2, fill] at (4,-1) {};
\node[draw, circle,scale=.2, fill] at (10,-1) {};

\draw (-6,0)--(9,0) (-7,1)--(-6,0)--(-7,-1) (-2,1)--(-3,0)--(-2,-1) (-1,1)--(0,0) (4,1)--(3,0)--(4,-1) (5,1)--(6,0) (10,1)--(9,0)--(10,-1);

\draw [red] plot [smooth] coordinates {(10,3) (0,0) (10,-3)};
\draw [blue] plot [smooth] coordinates {(-7,3) (3,0) (-7,-3)};

\node[red] at (4.5,2.2) {$T(\ra{e})$};
\node[blue] at (-1.5,2.2) {$T(\la{f})$};

\node at (7.5,0.5) {${\ra{f}\!}_1$};
\node at (-4.5,0.5) {${\ra{e}\!}_2$};

\end{tikzpicture}
\end{center}
\caption{The tree $\hat{T}$.}
\end{figure}
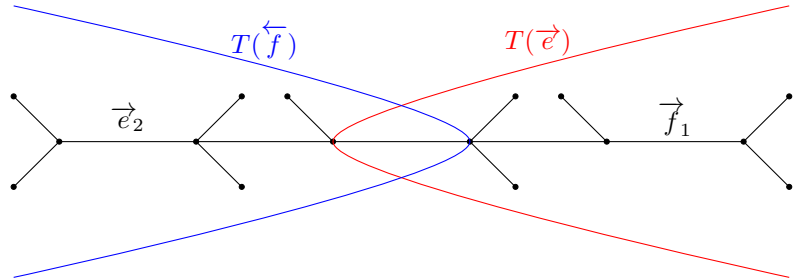
Note that, since $(X,Y)$ is linked to $(A,B)$ and $(Y,X)$ is linked to $(C,D)$ it follows from Lemma \ref{l:shift1} that for every $w \in T(\ra{e}) \cap T(\la{f})$
\[
|\alpha(\ra{w})| \geq \max \{ |\hat{\alpha}({\ra{w}\!}_1)|, |\hat{\alpha}({\ra{w}\!}_2)|\}.
\]

\begin{claim}\label{c:sepeq}
If $|\alpha(\ra{w})| >\ell$ and 
\[
|\alpha(\ra{w})| =|\hat{\alpha}({\ra{w}\!}_i)|
\]
then, 
\[
|\hat{\alpha}({\ra{w}\!}_{2-i})| \leq \ell.
\]
\end{claim}
\begin{remark}
For ease of exposition, we consider this to be vacuously satisfied if ${\ra{w}\!}_{2-i}$ doesn't exist.
\end{remark}
\begin{proof}
Indeed, suppose without loss of generality that $|\alpha(\ra{w})| =|\hat{\alpha}({\ra{w}\!}_1)|$ and that $\ra{e} \leq \ra{w}$. Note that, $\alpha(\ra{e}) = (A,B) \leq \alpha(\ra{w}) =: (E,F)$. Then,
\[
\hat{\alpha}({\ra{w}\!}_1) = (E,F)  \vee (X,Y).
\]
Since $|\hat{\alpha}({\ra{w}\!}_1)|=|E,F|$, by submodularity of the order function $|(E,F) \wedge (X,Y)| \leq |X,Y|$. However,
\[
(A,B) \leq (E,F) \wedge (X,Y)\leq (D,C)
\]
and so  $(E,F) \wedge (X,Y)$ was a candidate for $(X,Y)$. Moreover, since $(E,F)$ is nested with every separation in $\alpha\left(\ra{E}(T)\right)$, it follows from Lemma \ref{l:corner} that $(E,F) \wedge (X,Y)$ is nested with every separation in $\alpha\left(\ra{E}(T)\right)$ that $(X,Y)$ is, and also with $(E,F)$ itself. Hence, by our choice of $(X,Y)$, it follows that $(X,Y)$ was already nested with $(E,F)$. 

We may suppose that $w \in T(\ra{e}) \cap T(\la{f})$, since otherwise there is no ${\ra{w}\!}_2$. Hence, since $\ra{e} \leq \ra{w}$, there are two cases to consider, either $\ra{w} \leq \ra{f}$ or $\la{w} \leq \ra{f}$. Let us suppose first that $\ra{w} \leq \ra{f}$. Then, by definition, $\alpha_1({\ra{w}\!}_1) = (E,F) \vee (X,Y) = (E \cup X, F \cap Y)$ and $\alpha_2({\la{w}\!}_2) = (F,E) \vee (Y,X) = (F \cup Y, E \cap X)$.

There are now four cases as to how $(E,F)$ and $(X,Y)$ are nested. If $(E,F) \leq (X,Y)$ then $\alpha_1({\ra{w}\!}_1) = (X,Y)$, contradicting our assumption that $|\hat{\alpha}({\ra{w}\!}_1)| > \ell$. If $(F,E) \leq (Y,X)$ then $\alpha_2({\la{w}\!}_2) = (Y,X)$, and so $|\hat{\alpha}({\ra{w}\!}_2)| = \ell$.

If $(F,E) \leq (X,Y)$ then, since $\ra{S}$ is grounded, $|E \cup X, F \cap Y| = |\alpha_1({\ra{w}\!}_1)| \leq |X,Y| = \ell$, again a contradiction, and if $(E,F) \leq (Y,X)$ then $|F \cup Y, E \cap X| = |\alpha_2({\la{w}\!}_2)| \leq |X \cap Y| = \ell$.

Suppose then that $\la{w} \leq \ra{f}$. Again, $\alpha_1({\ra{w}\!}_1) = (E,F) \vee (X,Y) = (E \cup X, F \cap Y)$ and in this case $\alpha_2({\ra{w}\!}_2) = (E,F) \vee (Y,X) = (E \cup Y, F \cap X)$.

As before, there are four cases as to how $(E,F)$ and $(X,Y)$ are nested. If $(E,F) \leq (X,Y)$ then $\alpha_1({\ra{w}\!}_1) = (X,Y)$, a contradiction, and if $(E,F) \leq (Y,X)$ then $\alpha_2({\ra{w}\!}_2) = (Y,X)$, and $|\hat{\alpha}({\ra{w}\!}_2)| = \ell$.

If $(F,E) \leq (X,Y)$ then, since $\ra{S}$ is grounded, $|E \cup X, F \cap Y| = |\alpha_1({\ra{w}\!}_1)| \leq |X \cap Y| = \ell$, a contradiction, and finally if $(F,E) \leq (Y,X)$ then $|E \cup Y, F \cap X| = |\alpha_2({\ra{w}\!}_2)| \leq |X,Y| = \ell$.
\end{proof}

\begin{claim}\label{c:sepcontra}
For every $p > l$ and every $\ra{w} \in \ra{E}(T)$ with $|\alpha(\ra{w})| = p$ exactly one of $\hat{\alpha}({\ra{w}\!}_1),\hat{\alpha}({\ra{w}\!}_2)$ has order $p$, and the other has order $\leq \ell$. Furthermore, for each component $C$ of $T_p$ if $|\hat{\alpha}({\ra{w}\!}_i)| = |\alpha(\ra{w})|$ for some $\ra{w} \in C$, then $|\hat{\alpha}({\ra{w'}\!}_i)|= |\alpha(\ra{w'})|$ for every $\ra{w'} \in C$.
\end{claim}
\begin{proof}[Proof of Claim]
Suppose for contradiction that $p>l$ is the largest integer where the claim fails to hold. It follows that $e(\hat{T}_{p'}) = e(T_{p'})$ and $c(\hat{T}_{p'}) = c(T_{p'})$  for all $p' > p$. Hence, by $\prec$-minimality of $T$, $e(\hat{T}_p) \geq e(T_p)$. However, by assumption, for every separation $\alpha(\ra{w})$ of order $>p$ exactly one of $\hat{\alpha}({\ra{w}\!}_1), \hat{\alpha}({\ra{w}\!}_2)$ has the same order, and the other has order $\leq \ell$. Also, if $|\alpha(\ra{w})| = p > \ell$, then $|\alpha(\ra{w})| \geq \max \{ |\hat{\alpha}({\ra{w}\!}_1)|, |\hat{\alpha}({\ra{w}\!}_2)| \}$ and by Claim \ref{c:sepeq} if one of $\hat{\alpha}({\ra{w}\!}_1), \hat{\alpha}({\ra{w}\!}_2)$ has the same order as $\alpha(\ra{w})$ then the other has order $\leq \ell$. 

Hence, $e(\hat{T}_p) \leq e(T_p)$, and so by $\prec$-minimality of $T$ it follows that $e(\hat{T}_p) = e(T_p)$, and for every $\ra{w}$ with $|\alpha(\ra{w})|=p$ exactly one of $\hat{\alpha}({\ra{w}\!}_1), \hat{\alpha}({\ra{w}\!}_2)$ is of order $p$, and the other has order $\leq \ell$.

Recall that $\hat{T}$ was formed by joining a copy of $T(\ra{e})$ and $T(\la{f})$ along a separation $(X,Y)$ of order $\ell < p$. It follows from the first part of the claim that $c(\hat{T}_p)\geq c(T_p)$, and so by $\prec$-minimality of $T$, $c(\hat{T}_p)= c(T_p)$. Hence, for each component $C$ of $T_p$ if $|\hat{\alpha}({\ra{w}\!}_i)| = |\alpha(\ra{w})|$ for some $\ra{w} \in C$, then $|\hat{\alpha}({\ra{w'}\!}_i)|= |\alpha(\ra{w'})|$ for every $\ra{w'} \in C$. Therefore the claim holds for $p$, contradicting our assumption.
\end{proof}

By assumption, $\ra{e}$ and $\ra{f}$ lie in the same component of $T_{\ell+1}$, since for every $\ra{e} \leq \ra{g} \leq \ra{f}$
\[
|\alpha(\ra{g})| > |X,Y| = \ell.
\]
However, $\alpha_2({\ra{e}\!}_2) = (A,B)$ and $\alpha_1({\la{f}\!}_1) = (C,D)$, contradicting Claim \ref{c:sepcontra}.
\end{proof}

\section{$\cc{F}$-lean $S_k$-trees}\label{s:lean}
\begin{theorem}\label{t:lean}
Let $\ra{S}$ be a universe of set separations with an order function $|.|_r$ for some non-negative non-decreasing submodular function $r: 2^V \rightarrow \N$, $k \in \N$, and let $\cc{F} \subset \N^{{\ra{S}\!}_k}$ be a family of stars which is fixed under shifting. If there exists an $S_k$-tree over $\cc{F}$, then there exists an $\cc{F}$-lean $S_k$-tree over $\cc{F}$.
\end{theorem}
\begin{proof}
Let $(T,\alpha)$ be an $S_k$-tree over $\cc{F}$. Note that, since $\cc{F}$ is a family of stars, $(T,\alpha)$ is tame. We write $T^p$ for the induced subforest of $T$ on the set of vertices
\[
V(T^p) = \{ t \in T \, : \,  \langle \sigma_t \rangle_r \leq p \}
\]
Let us write $v(T^p)$ for the number of vertices of $T^p$ and $c(T^p)$ for the number of components.

We define an order on the set 
\[
\cc{T} = \{ (T,\alpha) \,: \, (T,\alpha) \text{ an } S_k \text{-tree over } \cc{F}\}
\]
as follows. We say that $(T,\alpha) \prec (S,\alpha')$ if there exists an $p \in \mathbb{N}$ such that for all $p'>p$ $v(T^{p'}) = v(S^{p'})$ and $c(T^{p'}) = c(S^{p'})$ and either:
\begin{itemize}
\item $v(T^p) < v(S^p)$, or
\item $v(T^p)=v(S^p)$ and $c(T^p) > c(S^p)$.
\end{itemize}
Let $(T,\alpha)$ be a $\prec$-minimal element of $\cc{T}$. We claim that $(T,\alpha)$ is $\cc{F}$-lean.

Suppose not, then there is some pair of vertices $t,t' \in T$ and a pair of separations $(A,B) \leq (B',A')$ such that $(A,B)$ is addable at $t$ and $(A',B')$ is addable at $t'$, with
\[
\lambda\big( (A,B) , (B,A) \big) <  \min \{ r(A), r(A') \}
\]
and
\[
|\alpha(\ra{g})| > \lambda\big( (A,B) , (B',A') \big) \text{ for all } \ra{g} \in t T t'.
\]

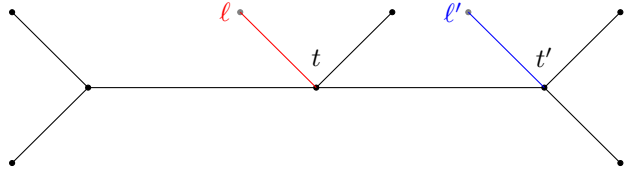
\begin{figure}[ht!]
\begin{center}
\begin{tikzpicture}

\node[draw, circle,scale=.2, fill] at (0,0) {};
\node[draw, circle,scale=.2, fill] at (-3,0) {};
\node[draw, circle,scale=.2, fill] at (3,0) {};
\node[draw, circle,scale=.2, fill] at (-4,-1) {};
\node[draw, circle,scale=.2, fill] at (-4,1) {};
\node[draw, circle,scale=.2, fill] at (1,1) {};
\node[draw, circle,scale=.2, fill] at (4,1) {};
\node[draw, circle,scale=.2, fill] at (4,-1) {};
\node[draw, gray, circle,scale=.2, fill] at (-1,1) {};
\node[draw, gray, circle,scale=.2, fill] at (2,1) {};

\draw (-3,0)--(3,0) (-4,1)--(-3,0)--(-4,-1) (1,1)--(0,0) (4,1)--(3,0)--(4,-1);
\draw[red] (0,0)--(-1,1);
\draw[blue] (3,0)--(2,1);

\node at (0,0.4) {$t$};
\node at (3,0.4) {$t'$};

\node[red] at (-1.2,1) {$\ell$};
\node[blue] at (1.8,1) {$\ell'$};

\end{tikzpicture}
\end{center}
\caption{The tree $T$.}
\end{figure}

Let $(T_1,\alpha_1)$ be the $S_k$-tree formed in the following way: $T_1$ is formed from $T$ by adding an extra leaf $\ell$ at $t$, $\alpha_1 \restriction \ra{E}(T) = \alpha$, and $\alpha_1(\ell,t) = (A,B)$. Similarly let $(T_2,\alpha_2)$ be the $S_k$-tree formed in the following way: $T_2$ is formed from $T$ by adding an extra leaf $\ell'$ at $t'$, $\alpha_2 \restriction \ra{E}(T) = \alpha$,  and $\alpha_2(\ell',t') = (A',B')$. Note that, since $(A,B)$ and $(A',B')$ were addable at $t$ and $t'$, $(T_1,\alpha_1)$ and $(T_2,\alpha_2)$ are tame $S_k$-trees over $\cc{F} \cup \{ (B,A) \} $ and $\cc{F} \cup \{(B',A')\}$ respectively. Let us denote by $\ra{e_1}$ the edge $(\ell,t) \in \ra{E}(T_1)$ and $\la{f_2}$ the edge $(\ell',t') \in \ra{E}(T_2)$. For each vertex and edge in $T$ let us write $v_1$ and $v_2$ or $e_1$ and $e_2$ for the copy of $v$ or $e$ in $T_1$ and $T_2$ respectively, and similarly $\sigma^i_{v_i}$ for the stars $\alpha_i({\ra{F}\!}_{v_i})$. Note that, for every $v \neq t,t'$ we have $\sigma^i_{v_i} = \sigma_v$, however $\sigma^1_{t_1} = \sigma_t \cup \{(A,B)\}$ and $\sigma^2_{t'_2} = \sigma_{t'} \cup \{(A',B')\}$.

Over all separations $(X,Y)$ such that $(A,B) \leq (X,Y) \leq (B',A')$ and $|X,Y|_r = \lambda\big( (A,B) , (B',A') \big) =: \ell$ we pick $(X,Y)$ which is nested with maximally many separations in $\alpha(\ra{E}(T))$. Note that $(X,Y)$ is linked to $(A,B)$ and $(Y,X)$ is linked to $(A',B')$.

Let $(T_1,\beta_1)$ be the $S_k$-tree given by shifting $(T_1,\alpha_1)$ onto $(X,Y)$ with respect to ${\ra{e}\!}_1$ and let $(T_2,\beta_2)$ be the $S_k$-tree given by shifting $(T_2,\alpha_2)$ onto $(Y,X)$ with respect to ${\la{f}\!}_2$. Note that, since $\cc{F}$ is fixed under shifting, by Lemmas \ref{l:shift1} and \ref{l:closedundershifting}, $(T_1,\alpha_1)$ and $(T_2,\beta_2)$ are $S_k$-trees over $\cc{F} \cup \{ (X,Y) \}$ and $\cc{F} \cup \{ (Y,X) \}$ respectively and that, since $\ell$ and $\ell'$ were leaves, these are indeed $S_k$ trees with underlying trees $T_1$ and $T_2$ respectively. We let $(\hat{T},\hat{\alpha})$ be the following $S_k$-tree:

$\hat{T}$ is the tree formed by taking the disjoint union of $T_1$ and $T_2$, and identifying the edge ${\ra{e}\!}_1$ with the edge ${\ra{f}\!}_2$. $\hat{\alpha}$ is then formed by taking the disjoint union of $\beta_1$ and $\beta_2$ on the domain $\ra{E}(\hat{T})$ 
where we note that, since $\beta_1({\ra{e}\!}_1) = (X,Y) = \beta_2({\ra{f}\!}_2)$ this map is well defined. We note that, by Lemma \ref{l:closedundershifting}, $(\hat{T},\hat{\alpha})$ is an $S_k$-tree over $\cc{F}$. For each vertex in $v \in T$ there are now two copies $v_1$ and $v_2$ in $\hat{T}$. We will write $\hat{\sigma}_{v_i}$ for the stars $\hat{\alpha}({\ra{F}\!}_{v_i})$. 

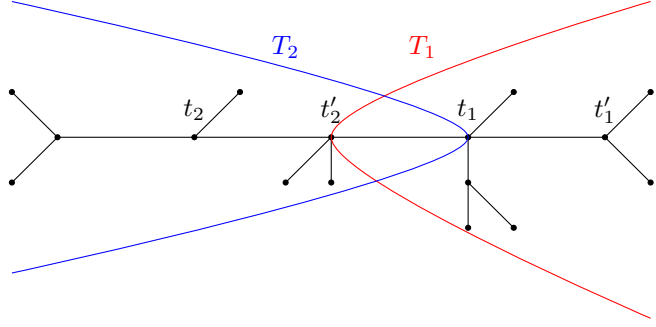
\begin{figure}[ht!]
\begin{center}
\begin{tikzpicture}[scale=0.6]

\node[draw, circle,scale=.2, fill] at (0,0) {};
\node[draw, circle,scale=.2, fill] at (-3,0) {};
\node[draw, circle,scale=.2, fill] at (3,0) {};
\node[draw, circle,scale=.2, fill] at (6,0) {};
\node[draw, circle,scale=.2, fill] at (-6,0) {};
\node[draw, circle,scale=.2, fill] at (-7,1) {};
\node[draw, circle,scale=.2, fill] at (-2,1) {};
\node[draw, circle,scale=.2, fill] at (4,1) {};
\node[draw, circle,scale=.2, fill] at (7,1) {};
\node[draw, circle,scale=.2, fill] at (-7,-1) {};
\node[draw, circle,scale=.2, fill] at (-1,-1) {};
\node[draw, circle,scale=.2, fill] at (0,-1) {};
\node[draw, circle,scale=.2, fill] at (3,-1) {};
\node[draw, circle,scale=.2, fill] at (7,-1) {};
\node[draw, circle,scale=.2, fill] at (3,-2) {};
\node[draw, circle,scale=.2, fill] at (4,-2) {};

\draw (-6,0)--(6,0) (-7,1)--(-6,0)--(-7,-1) (-2,1)--(-3,0) (0,-1)--(0,0)--(-1,-1) (4,1)--(3,0)--(3,-1)--(3,-2) (3,-1)--(4,-2) (7,1)--(6,0)--(7,-1);

\node at (0,0.6) {$t'_2$};
\node at (-3,0.6) {$t_2$};
\node at (3,0.6) {$t_1$};
\node at (6,0.6) {$t'_1$};

\draw [red] plot [smooth] coordinates {(7,3) (0,0) (7,-4)};
\draw [blue] plot [smooth] coordinates {(-7,3) (3,0) (-7,-3)};

\node[red] at (2,2) {$T_1$};
\node[blue] at (-1,2) {$T_2$};

\end{tikzpicture}
\end{center}
\caption{The tree $\hat{T}$.}
\end{figure}

Note that, since $(X,Y)$ is linked to $(A,B)$ and $(Y,X)$ is linked to $(A',B')$ it follows from Lemma \ref{l:shift2} that for every $s_i \in V(T_i)$
\begin{equation}\label{e:smaller}
\langle \hat{\sigma}_{s_i} \rangle_r \leq \langle \sigma^i_{s_i} \rangle_r \leq \langle \sigma_{s} \rangle_r,
\end{equation}
where the final inequality holds trivially for all $s \neq t,t'$, and for $s=t,t'$ since adding a separation to a star only decreases the size.

\begin{claim}\label{c:equal}
If $\langle \sigma_{s} \rangle_r > \ell$ and
\[
\langle \hat{\sigma}_{s_i} \rangle_r = \langle \sigma_{s} \rangle_r
\]
then
\[
\langle \hat{\sigma}_{s_{i-2}} \rangle_r \leq \ell.
\]
\end{claim}
\begin{proof}
Indeed, let us assume without loss of generality that 
\[
\langle \hat{\sigma}_{s_1} \rangle_r = \langle \sigma^1_{s_1} \rangle_r,
\]
and suppose first that $s \neq t,t'$, and so $\sigma^i_{s_i} = \sigma_s$. Let us write
\[
\sigma_s = \{ (A_0,B_0), \ldots, (A_n,B_n) \}.
\]
For each $s \in T$ there is a unique edge $\ra{g}$ such that ${\ra{e}\!}_1 \leq {\ra{g}\!}_1$, and similarly a unique edge $\ra{h}$ such that ${\la{f}\!}_2 \leq {\ra{h}\!}_2$. Let us suppose without loss of generality that $\alpha(\ra{g}) = (A_0,B_0)$ and $\alpha(\ra{h}) = (A_j,B_j)$, where perhaps $j=0$. Then
\[
\hat{\sigma}_{s_1} = \{ (A_0 \cup X,B_0 \cap Y), \ldots, (A_n \cap Y,B_n \cup X) \}
\]
and 
\[
\hat{\sigma}_{s_2} =  \{ (A_j \cup Y,B_j \cap X), \ldots, (A_n \cap X,B_n \cup Y) \}
\]

Now, since $\langle \hat{\sigma}_{s_1} \rangle_r \leq \langle \sigma_{s} \rangle_r$, it follows that equality holds throughout the proof of Lemma \ref{l:shift2}, and so in particular
\[
|B^* \cap X, A^* \cup Y|_r = |X,Y|_r,
\]
where  $B^* =  \bigcap_{i\neq 0} B_i$ and $A^* = \bigcup_{i \neq 0} A_i$.

However, $(A,B) \leq (B^* \cap X, A^* \cup Y) \leq (B',A')$ and so, by our choice of $(X,Y)$, $(X,Y)$ is nested with at least as many separations in $\alpha(\ra{E}(T))$ as $(B^* \cap X, A^* \cup Y)$. However, since $(B^*,A^*)$ is nested with $\alpha(\ra{E}(T))$ and $(B^* \cap X, A^* \cup Y) \leq (B_i,A_i)$ for each $i \neq 0$, it follows by Lemma  \ref{l:corner} that $(X,Y)$ was already nested with $(B_i,A_i)$ for each $i \neq 0$.

We note that if $(B_i,A_i) \leq (Y,X)$ for any $i \neq 0$, then $(A,B) \leq (B_i,A_i) \leq (Y,X)$ and $(A,B) \leq (X,Y)$. Hence $A \subset X \cap Y$ and so $r(A) \leq r(X \cap Y) \leq |X,Y|_r$, contradicting our assumptions on $A$. Similarly if $(B_i,A_i) \leq (X,Y)$ for $i \neq j$.

Therefore, for each $i \neq j,0$, either $(X,Y) \leq (B_i,A_i)$ or $(Y,X) \leq (B_i,A_i)$. Note that, in the first case  $r(B_i \cup Y) = r(V)$ and $r(B_i \cup X) = r(X)$, and vice versa. Hence,
\begin{align*}
\langle \hat{\sigma}_{s_1} \rangle_r + \langle \hat{\sigma}_{s_2} \rangle_r &= r(B_0 \cap Y) + r(B_j \cap X) + \sum_{i \neq 0} r(B_i \cup X) + \sum_{i \neq j} r(B_i \cup Y) - 2n.r(V) \\
&=  r(B_0 \cap Y) + r(B_j \cap X) + r(B_0 \cup Y) + r(B_j \cup X) + \sum_{i \neq 0,j} r(B_i)  - (n+1).r(V) \\
&= \langle \sigma_{s} \rangle_r + r(B_0 \cap Y) + r(B_j \cap X) + r(B_0 \cup Y) - r(B_0) - r(B_j) - r(V)\\
&\leq \langle \sigma_{s} \rangle_r + r(Y) + r(X) - r(V)\\
&\leq \langle \sigma_{s} \rangle_r + |X,Y|_r.
\end{align*}

From which is follows that, if $\langle \hat{\sigma}_{s_1} \rangle_r = \langle \sigma_{s} \rangle_r$, $ \langle \hat{\sigma}_{s_2} \rangle_r \leq |X,Y|_r$. If $s=t$ or $t'$, then a similar calculation holds.
\end{proof}

\begin{claim}\label{c:size}
\[
\langle \hat{\sigma}_{t_1} \rangle_r  < \langle \sigma_t \rangle_r.
\]
and
\[
\langle \hat{\sigma}_{t'_2} \rangle_r< \langle \sigma_{t'} \rangle_r.
\]
\end{claim}
\begin{proof}
Let us write
\[
\sigma_t = \{ (A_1,B_1), \ldots, (A_n,B_n) \},
\]
and so
\[
\langle \hat{\sigma}_{t_1} \rangle_r = r(Y) + \sum_{i=1}^ n r(B_i \cup X) - n.r(V).
\]
Suppose for contradiction that $\langle \hat{\sigma_{t_1}} \rangle_r = \langle \sigma_t \rangle_r$. As before it follows that we can split $(A_i,B_i)$ into sets $I$ and $I'$ such that $(X,Y) \leq (B_j,A_j)$ for $j \in I $ and $(Y,X) \leq (B_j,A_j)$ for $j \in I'$. It follows that
\[
\langle \hat{\sigma_{t_1}} \rangle_r = r(Y) + \sum_{j \in I} r(B_j) - |I|.r(V).
\]
However, since $\sigma_t$ is a star it follows by Claim \ref{c:starint} that
\[
\sum_{j \in I'} r(B_j) \geq r(\bigcap_{j \in I'} B_j) + (|I'|-1).r(V).
\]
Now, $A \subset \bigcap_{i \in I'} B_i$ and also by definition of $I'$, $Y \subset \bigcap_{j \in I'} B_j$ and hence $r(Y \cup A) \leq r(\bigcap_{j \in I'} B_j)$. However, by submodularity
\[
r(X) + r(A \cup Y) \geq r(A) + r(V)
\]
and by assumption $r(A) > |X,Y|_r = r(X) + r(Y) - r(V)$, and so
\[
r(A \cup Y) > r(Y),
\]
from which we conclude that
\begin{align*}
r(Y) &< r(Y \cup A) \\
&\leq r(\bigcap_{j \in I'} B_j) \\
&\leq \sum_{j \in I'} r(B_j) - (|I'|-1).r(V).
\end{align*}
Therefore,
\begin{align*}
\langle \hat{\sigma}_{t_1} \rangle_r &= r(Y) + \sum_{j \in I} r(B_j) - |I|.r(V) \\
&< \sum_{i=1}^n r(B_i) - (n-1).r(V) = \langle \sigma_t \rangle_r.
\end{align*}
A similar argument shows that $\langle \hat{\sigma}_{t'_2} \rangle_r < \langle \sigma_{t'} \rangle_r$.
\end{proof}

\begin{claim}\label{c:alleq}
For every $p > \ell$ and every $s \in V(T)$ with $\langle \sigma_s \rangle_r = p$ exactly one of $\hat{\sigma}_{s_1} , \hat{\sigma}_{s_2}$ has size $p$, and the other has size $\leq \ell$ and further for each component $C$ of $T^p$, if $\langle \hat{\sigma}_{s_i} \rangle_r =\langle \sigma_s \rangle_r$ for some $s \in C$, then $\langle \hat{\sigma}_{s'_i} \rangle_r =\langle \sigma_{s'} \rangle_r$ for every $s' \in C$.
\end{claim}
\begin{proof}
Let $\overline{p} = \max \{ \langle \sigma_s \rangle_r \, : \, s \in V(T)\}$. We will prove the statement by induction starting with $p=\overline{p}$.

By $\prec$-minimality of $T$ it follows that $v(\hat{T}^{\overline{p}}) \geq v(T^{\overline{p}})$. However, since $\overline{p} \geq \langle \sigma_t \rangle_r \geq r(A) >\ell$, for every $s \in V(T)$ with $\langle \sigma_s \rangle_r = \overline{p}$, by Claim \ref{c:equal} if one of $\hat{\sigma}_{s_1} , \hat{\sigma}_{s_2} $ has the same size as $\sigma_s$, the other has size $\leq \ell$.

Therefore, it follows that $v(\hat{T}^{\overline{p}}) \leq v(T^{\overline{p}})$, and so $v(\hat{T}^{\overline{p}}) = v(T^{\overline{p}})$, and for every $s \in V(T)$ with $\langle \sigma_s \rangle_r = \overline{p}$ exactly one of $ \hat{\sigma}_{s_1} , \hat{\sigma}_{s_2} $ has the same size as $\sigma_s$, and the other has size $\leq \ell$. Recall that $\hat{T}$ is formed by joining two copies of $T$ by an edge between $t_1$ and $t'_2$, and that by Claim \ref{c:size} $\hat{\sigma}_{t_1}$ and $\hat{\sigma}_{t'_2}$ both have size $<\overline{p}$. It follows that $c(\hat{T}^{\overline{p}}) \geq c(T^{\overline{p}})$, and so again by $\prec$-minimality of $T$, $c(\hat{T}^{\overline{p}}) = c(T^{\overline{p}})$. Furthermore, for each component $C$ of $T^{\overline{p}}$, if $\langle \hat{\sigma}_{s_i} \rangle_r =\langle \sigma_s \rangle_r$ for some $s \in C$, then $\langle \hat{\sigma}_{s'_i} \rangle_r =\langle \sigma_{s'} \rangle_r$ for every $s' \in C$.

Let $\ell < p < \overline{p}$ and suppose the claim holds for all $p'>p$, . We split into three cases, first let us suppose $p \geq \max \{ \langle \sigma_t \rangle_r, \langle \sigma_{t'} \rangle_r \}$. In this case, as before, by $\prec$-minimality of $T$ $v(\hat{T}^p) \geq v(T^p)$. However, by the induction claim, for every $s \in V(T)$ with $\langle \sigma_s \rangle_r > p$ exactly one of $\hat{\sigma}_{s_1} , \hat{\sigma}_{s_2}$ has size $p$, and the other has size $\leq \ell$. Furthermore, by Claim \ref{c:equal}, for every $s \in V(T)$ with $\langle \sigma_s \rangle_r = p$ if one of $\hat{\sigma}_{s_1} , \hat{\sigma}_{s_2} $ has the same size as $\sigma_s$, the other has size $\leq \ell$.

Therefore it follows that $v(\hat{T}^p) \leq v(T^p)$, and so $v(\hat{T}^p) = v(T^p)$, and for every $s \in V(T)$ with $\langle \sigma_s \rangle_r = p$ exactly one of $ \hat{\sigma}_{s_1} , \hat{\sigma}_{s_2} $ has the same size as $\sigma_s$, and the other has size $\leq \ell$. As before, recall that $\hat{T}$ is formed by joining two copies of $T$ by an edge between $t_1$ and $t'_2$, and that by Claim \ref{c:size} $\hat{\sigma}_{t_1}$ and $\hat{\sigma}_{t'_2}$ both have size $< \max \{ \langle \sigma_t \rangle_r, \langle \sigma_{t'} \rangle_r \} \leq p$. It follows that $c(\hat{T}^p) \geq c(T^p)$, and so again by $\prec$-minimality of $T$, $c(\hat{T}^p) = c(T^p)$. Further, for each component $C$ of $T^p$, if $\langle \hat{\sigma}_{s_i} \rangle_r =\langle \sigma_s \rangle_r$ for some $s \in C$, then $\langle \hat{\sigma}_{s'_i} \rangle_r =\langle \sigma_{s'} \rangle_r$ for every $s' \in C$.

Suppose then that $p=\max \{ \langle \sigma_t \rangle_r, \langle \sigma_{t'} \rangle_r \}$, say without loss of generality $p= \langle \sigma_t \rangle_r$. As before we conclude that for every $s \in V(T)$ with $\langle \sigma_s \rangle_r = p$ exactly one of $ \hat{\sigma}_{s_1} , \hat{\sigma}_{s_2} $ has the same size as $\sigma_s$, and the other has size $\leq \ell$. In this case, since by Claim \ref{c:size} $\langle \hat{\sigma}_{t_1} \rangle_r < \langle \sigma_t \rangle_r = p$, it follows that $\langle \hat{\sigma}_{t_1} \rangle_r \leq \ell$. This allows us to conclude that the copies of each component of $T^p$ in $\hat{T}^p$ are separated by the vertex $t_1$, and thus $c(\hat{T}^p) \geq c(T^p)$. So as before we can conclude that for each component $C$ of $T^p$, if $\langle \hat{\sigma}_{s_i} \rangle_r =\langle \sigma_s \rangle_r$ for some $s \in C$, then $\langle \hat{\sigma}_{s'_i} \rangle_r =\langle \sigma_{s'} \rangle_r$ for every $s' \in C$.

Finally, when $\ell < p < \max \{ \langle \sigma_t \rangle_r, \langle \sigma_{t'} \rangle_r \}$ the same argument will hold, since one of $\langle \hat{\sigma}_{t_1} \rangle_r$ or $\langle \hat{\sigma}_{t_2} \rangle_r$ has size $ \leq \ell$.
\end{proof}

Suppose that $t=t'$. Since $(A,B)$ and $(A',B')$ are addable at $\sigma_t$ it follows by Lemma \ref{l:sepstar} that $\langle \sigma_t \rangle_r \geq r(A) \geq \min \{r(A),r(A')\} > \ell$, and so by Claim \ref{c:alleq}
\[
\langle \sigma_{t} \rangle_r = \max \{ \langle \hat{\sigma}_{t_1} \rangle_r, \langle \hat{\sigma}_{t_2} \rangle_r  \},
\]
contradicting Claim \ref{c:size}.

Suppose that $t \neq t'$. Then, since 
\[
|\alpha(\ra{g})| > \ell \text{ for all } \ra{g} \in t T t,
\]
it follows by Lemma \ref{l:sepstar} that $\langle \sigma_s \rangle_r >  \ell$ for all $s \in tTt'$, and so $t$ and $t'$ live in the same component of $T^{(\ell +1)}$. However, $\langle \hat{\sigma}_{t_1} \rangle_r < \ell$ and $\langle \hat{\sigma}_{t'_2} \rangle_r < \ell$, contradicting Claim \ref{c:alleq}.

\end{proof}

\begin{remark}
Unlike in the case of tree-width for graphs it is not true in general that the existence of a $\cc{F}$-lean $S_k$-tree over $\cc{F}$ implies the existence of a linked $S_k$-tree over $\cc{F}$. It seems likely that similar methods should prove the existence of an $S_k$-tree which is both linked and $\cc{F}$-lean. However, since Theorem \ref{t:link} could be stated in slightly more generality, and to avoid lengthening an already quite technical proof, we have stated the results separately. Explicitly, one should consider a minimal element of the following order on the set of $S_k$-trees over $\cc{F}$: we let $T<S$ if there is some $p \in \mathbb{N}$ such that:
\begin{itemize}
\item $e(T_p) < e(S_p)$; or
\item $e(T_p) = e(S_p)$ and $c(T_p) > c(S_p)$; or
\item $e(T_p) = e(S_p), c(T_p) > c(S_p)$ and $v(T^p) < v(S^p)$; or
\item $e(T_p) = e(S_p), c(T_p) > c(S_p), v(T^p) < v(S^p)$ and $c(T^p) > c(S^p)$,
\end{itemize}
and for all $p' > p$, all four quantities are equal. 
\end{remark}

\section{Applications}\label{s:app}
\subsection{Graphs}
Throughout this subsection $\ra{S}$ is the universe of separations of some graph. Given a star 
\[
\sigma = \{ (A_0,B_0), (A_1,B_1), \ldots , (A_n,B_n) \}
\]
we define int$(\sigma) = \bigcap_{i=0}^n B_i$. Given an $S_k$ tree $(T,\alpha)$ we can construct an associated tree-decompositon $(T,\cc{V})_{\alpha}$ by letting
\[
V_t = \text{int}(\sigma_t) \text{ for each } t \in T.
\]
\begin{lemma}\label{l:build}
If $(T,\alpha)$ is a tame $\ra{S}$-tree then $(T,\cc{V})_{\alpha}$ is a tree-decompositon.  Furthermore, if $\alpha(t,t') = (A,B)$ then $V_t \cap V_{t'}= A \cap B$.
\end{lemma}
\begin{proof}
Given an edge $e \in E(G)$ we can define an orientation on $E(T)$ as follows. For each edge $f \in E(T)$ let us pick one of the orientations $\ra{f}$ such that $\alpha(\ra{f}) = (A,B)$ with $e \in B$. Note that, since $\alpha(\ra{f})=(A,B)$ is a separation, either $e \in A$ or $e \in B$. Any orientation has a sink $t$, and it is a simple check that $e \in V_t$. This proves that $(T,\cc{V})$ satisfies the first two properties of a tree-decomposition.

Finally suppose that $t_2 \in t_1Tt_3$ and $x  \in V_{t_1} \cap V_{t_3}$. There is a unique edge $\ra{e} \in \ra{F}_{t_2}$ such that $t_1 \in T(\la{e})$ and similarly a unique edge $\ra{f} \in \ra{F}_{t_2}$ such that $t_3 \in T(\la{f})$. Let $\alpha(\ra{e}) = (A,B)$ and $\alpha(\ra{f}) = (C,D)$ and let us write
\[
\sigma_{t_2} = \{ (A,B), (C,D), (A_1,B_1), \ldots, (A_n,B_n) \}.
\]
Since $x \in V_{t_1}$ and $\alpha$ preserves the tree-ordering, it follows that $x \in A$ and hence $x \in D \cap \bigcap_{i=1}^n B_i$. Similarly, since $x \in V_{t_3}$, $x \in C$ and hence $x \in B \cap \bigcap_{i=1}^n B_i$ and so 
\[
x \in B \cap D \cap \bigcap_{i=1}^n B_i = \text{int}(\sigma_{t_2}) = V_{t_2}.
\]

Finally, suppose that $\alpha(t,t') = (A,B)$. Since $V_t = \text{int}(\sigma_t) \subseteq A$ and $V_{t'} = \text{int}(\sigma_{t'}) \subseteq B$ it follows that $V_t \cap V_{t'} \subseteq A \cap B$. Conversely it is a simple check that, since $\sigma_t$ and $\sigma_{t'}$ are stars, $A \cap B \subset \text{int}(\sigma_t)$ and $A \cap B \subset \text{int}(\sigma_{t'})$. It follows that $V_t \cap V_{t'}= A \cap B$
\end{proof}
We note that, conversely, given a tree-decomposition $(T,\cc{V})$ we can build a tame $S_k$-tree $(T,\alpha)_{\cc{V}}$ by letting 
\[
\alpha(\ra{e}) = \alpha(t,t') = \left( \bigcup_{s \, : \, t \in sTt'} V_s, \bigcup_{s \, : \, t'\in sTt} V_s\right),
\]
and in this way the two notions are equivalent.

In this way, by applying Theorems \ref{t:link} and \ref{t:lean} to appropriate families $\cc{F}$ of stars we can prove a number of results about tree-decompositions of graphs, some known and some new. We will adapt Definition \ref{d:link} slightly to the following

\begin{defn}\label{d:link2}
A tree decomposition $(T,\cc{V})$ is called \emph{linked} if for all $k \in \bb{N}$ and every $t,t' \in T$, either $G$ contains $k$ disjoint $V_{t}$-$V_{t'}$ paths or there is an edge $(s,s') \in tTt'$ such that $|V_s \cap V_{s'}|<k$.
\end{defn}

As in the introduction, it is easy to see that if a linked tree-decomposition exists in the sense of Definition \ref{d:link2}, then by subdividing each edge and adding as a bag the separating set $V_t \cap V_{t'}$ we obtain a linked tree-decomposition in the sense of Definition \ref{d:link}.

\begin{lemma}\label{l:linktree}
Let $\ra{S}$ be the universe of graph separations for some graph $G$. If $(T,\alpha)$ is a linked tame $S_k$-tree then $(T,\cc{V})_{\alpha}$ is linked.
\end{lemma}
\begin{proof}
Suppose that $(T,\alpha)$ is a linked tame $S_k$-tree. Given $r \in \mathbb{N}$ and $t$ and $t' \in T$ such that $G$ does not contain $r$ disjoint $V_{t}$-$V_{t'}$ paths, we wish to show that there is an edge $(s,s') \in tTt'$ such that $|V_s \cap V_{s'}|<r$. Let $\ra{e}$ be the unique edge adjacent to $t$ such that $t' \in T(\ra{e})$ and similarly let $\la{f}$ be the unique edge adjacent to $t'$ such that $t \in T(\la{f})$. Note that $\ra{e} \leq \ra{f}$.

Let us write $\alpha(\ra{e}) = (A,B)$ and $\alpha(\ra{f}) = (C,D)$. Since $A \cap B \subset V_{t}$ and $C \cap D \subset V_{t'}$, and $G$ does not contain $r$ disjoint $V_{t}-V_{t'}$ paths, it follows by Menger's theorem that $\lambda\left( (A,B), (C,D)\right) < r$. Hence, since $(T,\alpha)$ is linked, there is some edge $\ra{e} \leq \ra{g} \leq \ra{f}$ such that $|\alpha(\ra{g})| =: |X,Y| < r$. Let us write $\ra{g} = (s,s')$. Note that, $(s,s') \in tTt'$ by construction, and by Lemma \ref{l:build} $|V_s \cap V_{s'}| = |X \cap Y| < r$, as claimed.
\end{proof}

For many families of stars, being $\cc{F}$-lean will not tell us much about the tree-decomposition. Indeed, if $\cc{F}$ only contains multi-sets of size $3$ or $1$ (as in the case of branch decompositions), then there is never an addable separation for any star in $\cc{F}$. However, for certain families of stars being $\cc{F}$-lean will imply leaness in the traditional sense. 

\begin{defn}
Let $\ra{S}$ be a separation system. A family of stars $\cc{F} \subset 2^{\ra{S}}$ is \emph{$S$-stable} if whenever $\sigma \in \cc{F}$ and $(A,B) \in \ra{S}$ is such that $\sigma \cup \{(A,B)\}$ is a star then $\sigma \cup \{(A,B)\} \in \cc{F}$.
\end{defn}

\begin{lemma}\label{l:stable}
Let $\ra{S}$ be a universe of separations with an order function $|.|_r$ for some non-decreasing submodular function $r: 2^V \rightarrow \mathbb{N}$, and let $p \in \mathbb{N}$. Then $\cc{F}_p$ is $S$-stable.
\end{lemma}
\begin{proof}
If $(A,B) \in \ra{S}$ is such that $\sigma \cup \{(A,B) \}$ is a star then,
\[
\langle  \sigma \cup \{(A,B)\} \rangle_r = \langle \sigma \rangle_r + r(B) - r(V) \leq \langle \sigma \rangle_r < p,
\]
and so $\sigma \cup \{(A,B)\}  \in \cc{F}_p$.
\end{proof}

\begin{lemma}\label{l:leantree}
Let $\ra{S}$ be the universe of graph separations for some graph $G$ and let $\cc{F} \subset \N^{\ra{S}}$ be an $S$-stable family of stars. If $(T,\alpha)$ is an $\cc{F}$-lean tame $S$-tree then $(T,\cc{V})_{\alpha}$ is lean.
\end{lemma}
\begin{proof}
Given $k \in \mathbb{N}$, $t,t'\in T$ and vertex sets $Z_1 \subseteq V_{t}$ and $Z_2 \subseteq V_{t'}$ with $|Z_1|=|Z_2|=k$, such that $G$ doesn't contain $k$ disjoint $Z_1$-$Z_2$ paths, we wish to show that there exists an edge $(s,s') \in tTt'$ with $|V_s \cap V_{s'}| < k$.

Since $Z_1 \subset V_{t}$ it follows that $\{(Z_1,V) \} \cup \sigma_{t}$ forms a star, and similarly so does $\{(Z_2,V) \} \cup \sigma_{t'}$. Since $\cc{F}$ is $\ra{S}$-stable, both of these stars are in $\cc{F}$, and so $(Z_1,V)$ is addable at $\sigma_{t}$ and $(Z_2,V)$ is addable at $\sigma_{t'}$. Since $G$ does not contain $k$ disjoint $Z_1$-$Z_2$ paths, it follows that $\lambda \left( (Z_1,V),(V,Z_2) \right) < k$. Hence, since $(T,\alpha)$ is $\cc{F}$-lean, there exists an edge $\ra{g} \in tTt'$ such that $|\alpha(\ra{g})|:= |X,Y| = \lambda \left( (Z_1,V),(V,Z_2) \right)$. As before let us write $\ra{g} = (s,s')$. Then $(s,s') \in tTt'$ and $|V_s \cap V_{s'}| = |X \cap Y| < k$, as claimed.
\end{proof}

Given a star $\sigma \in \N^{\ra{S}}$ let us write $n(\sigma)$ for the cardinality of the multiset $\sigma$. Diestel and Oum \cite{DO142} showed that for the families of stars
\[
\cc{F}_k := \{ \sigma \in \N^{{\ra{S}\!}_k} \, : \, \sigma \text{ a star, } \langle \sigma \rangle_r <k\},
\]
\[
\cc{P}_k  := \{ \sigma \subset \cc{F}_k \, : \, n(\sigma) \leq 2 \}
\]
and 
\begin{align*}
\cc{T}_k := &\{ \sigma \in \N^{{\ra{S}\!}_k} \, : \, \sigma \text{ a star}, \sigma = \{ (A_1,B_1), (A_2,B_2), (A_3,B_3) \} \text{ and } \bigcup_{i=1}^3 G[A_i] = G \} \\
&\cup  \{ \sigma \in \cc{F}_k \, : \, n(\sigma) = 1 \},
\end{align*}
the following is true:
\begin{itemize}
\item $G$ admits a tree-decomposition of width $<k-1$ if and only if there is an $S_k$-tree over $\cc{F}_k$,
\item $G$ admits a path-decomposition of width $<k-1$ if and only if there is an $S_k$-tree over $\cc{P}_k$,
\item $G$ admits a branch-decomposition of width $<k$ if and only if there is an $S_k$-tree over $\cc{T}_k$.
\end{itemize}
Furthermore, they showed that in every case the family of stars was fixed under shifting \footnote{Strictly they showed that a slightly different condition, which they called \emph{closed under shifting}, holds, however the same proof shows they are also fixed under shifting.}. Also, they showed that if $\theta \leq p$ then $\cc{F}^{\theta}_p := \cc{F}_p \cap \N^{{\ra{S}\!}_\theta}$ is fixed under shifting, and that
\begin{itemize}
\item $G$ has $\theta$-tree-width $<p-1$ if and only if there is an $S_\theta$-tree over $\cc{F}^{\theta}_p$,
\end{itemize}
where the $\theta$-tree-width of $G$ $\theta$-tw$(G)$ is defined, as in \cite{GJ16}, to be the smallest $p$ such that $G$ admits a tree-decomposition of width $p$ and adhesion $<\theta$.

\begin{defn}
A tree decomposition $(T,\cc{V})$ is called \emph{$\theta$-lean} if for all $k <\theta$, $t,t' \in T$ and vertex sets $Z_1 \subseteq V_{t_1}$ and $Z_2 \subseteq V_{t_2}$ with $|Z_1|=|Z_2|=k$, either $G$ contains $k$ disjoint $Z_1$-$Z_2$ paths or there exists an edge $(s,s') \in tTt'$ with $V_s \cap V_{s'} < k$.
\end{defn}

We note that by the same arguments as Lemmas \ref{l:stable} and \ref{l:leantree}, if an $S_\theta$-tree over $\cc{F}^{\theta}_p$ is $\cc{F}^{\theta}_p$-lean, then the corresponding tree-decomposition is $\theta$-lean. Let us write pw$(G)$, bw$(G)$, and $\theta$-tw$(G)$ for the path-, branch-, and $\theta$-tree-width of $G$ respectively. Applying Theorems \ref{t:link} and \ref{t:lean} to these families of stars gives the following theorem:

\begin{theorem}
Let $G$ be a graph then the following are true:
\begin{itemize}
\item $G$ admits a linked path-decomposition of width pw$(G)$;
\item $G$ admits a linked branch-decomposition of width bw$(G)$;
\item $G$ admits a lean tree-decomposition of width tw$(G)$;
\item $G$ admits a $\theta$-lean tree-decomposition of width $\theta$-tw$(G)$.
\end{itemize}
\end{theorem}

We note that none of these results are in essence new. We don't know if the first result appears explicitly in the literature anywhere, however it is implied by the corresponding result for directed path width [\cite{K14}, Theorem 7]. The second result is implied by a broader theorem [\cite{GGW02}, Theorem 2.1] on linked branch decompositions of submodular functions and the third is Thomas' theorem, Theorem \ref{t:thomlink} in this paper. The fourth is stated without proof in [\cite{CDHH14}, Theorem 2.3], and a slightly weaker result is claimed in [\cite{GJ16}, Theorem 3.1] although the proof in fact shows the stronger statement. 

As mentioned before, there is never an addable separation for any star $\sigma \in \cc{T}_k$, and so Theorem \ref{t:lean} give us no insight into branch decompositions. However, for path-decompositions it tells us something about the bags at the two leaves.

\begin{lemma}
Let $(T,\alpha)$ be an $\cc{P}_k$-lean $S_k$-tree over $\cc{P}_k$ and let $T$ be a path $t_0,t_1, \ldots, t_n$. Then $(T,\cc{V})_{\alpha}$ has the following properties:
\begin{itemize}
\item For all $Z_1, Z_2 \subset V_{t_0}$ or $Z_1,Z_2 \subset V_{t_n}$ with $|Z_1| = |Z_2|=r$ there are $r$-disjoint $Z_1-Z_2$ paths in $G$.
\item For all $Z_1 \subset V_{t_0}$ and $Z_2 \subset V_{t_n}$ with $|Z_1| = |Z_2|=r$ either there are $r$ disjoint $Z_1-Z_2$ paths in $G$ or there is an edge $(t_i,t_{i+1}) \in T$ with $|V_{t_i} \cap V_{t_{i+1}}| < r$.
\end{itemize}
\end{lemma}
\begin{proof}
Let us show the first statement, the proof of the second is similar. Suppose $Z_1, Z_2 \subset V_{t_0}$. Since $Z_i \subset V_{t_0}$ it follows that $|Z_i| = r < k$ and so $(Z_i,V) \in {\ra{S}\!}_k$. Let $\alpha(t_1,t_0) = (A,B)$. Since $\{(A,B)\} \in \cc{P}_k$, $|B| <k$  and since $B = V_{t_0}$ and $Z_i \subseteq V_{t_0}$ both $(Z_i,V)$ are addable at $\{(A,B)\}$.

Therefore, since $(T,\alpha)$ is $\cc{P}_k$-lean, it follows that $\lam{(Z_0,V),(V,Z_1)} \geq \min \{ |Z_i| \} = r$. Hence $G$ contains $r$ disjoint $Z_1$-$Z_2$ paths.
\end{proof}

Not only is this result broad enough to imply many known theorems, the framework is also flexible enough to encompass many other types of tree-decompositions. For example, more recently, Diestel, Eberenz and Erde \cite{DEE15} showed that there exist families of stars $\cc{B}_k$ and $\cc{P}_k \subset \N^{{\ra{S}\!}_k}$, which are fixed under shifting, such that the existence of $S_k$-trees over $\cc{B}_k$ or $\cc{P}_k$ is dual to the existence of a $k$-block or $k$-profile in the graph respectively ($k$-blocks and $k$-profiles are examples of what Diestel and Oum call ``highly cohesive structures'' which represent obstructions to low width, see \cite{DO141}). They defined the \emph{profile-width} and \emph{block-width} of a graph $G$, which we denote by blw$(G)$ and prw$(G)$, to be the smallest $k$ such that there is an $S_k$-tree over $\cc{B}_k$ or $\cc{P}_k$ respectively. Again, applying Theorem \ref{t:link} to these families of stars we get the following theorem:

\begin{theorem}
Let $G$ be a graph then the following are true:
\begin{itemize}
\item $G$ admits a linked profile-decomposition of width prw$(G)$;
\item $G$ admits a linked block-decomposition of width blw$(G)$.
\end{itemize}
\end{theorem}

The family $\cc{B}_k$ of stars is built from a family $\cc{B}^*_k$ of stars by iteratively taking all possible stars that appear as shifts of stars in $\cc{B}^*_k$, in order to guarantee that $\cc{B}_k$ is fixed under shifting. The set $\cc{B}^*_k$ can be taken to be stable, but it is not clear if this property is maintained when moving to $\cc{B}_k$. It would be interesting to know if a lean block-decomposition of width blw$(G)$ always exists.

\subsection{Matroid tree-width}\label{s:mat}
Hlin\v{e}n\'{y} and Whittle \cite{HW06,HW09} generalized the notion of tree-width from graphs to matroids. Let $M=(E,I)$ be a matroid with rank function $r$. Hlin\v{e}n\'{y} and Whittle defined a tree-decomposition of $M$ to be a pair $(T,\tau)$ where $\tau : E \rightarrow V(T)$ is an arbitrary map. Every vertex $v \in V(T)$ separates the tree into a number of components $T_1,T_2,\ldots,T_d$ and we define the \emph{width} of the bag $\langle \tau^{-1}(v) \rangle$ to be 
\[
\sum_{i=1}^d r(E \setminus \tau^{-1}(T_i)) - (d-1).r(E).
\]
The \emph{width} of a tree-decomposition is $\max \{ \langle \tau^{-1}(v) \rangle \, : \, v \in V(T) \}$ and the \emph{tree-width} of $M$ is the smallest $k$ such that $M$ has a tree-decomposition of width $k$. This is a generalisation of the tree-width of graphs, and in particular Hlin\v{e}n\'{y} and Whittle showed that for any graph $G$ with at least one edge, if $M(G)$ is the cycle matroid of $G$ then the tree-width of $G$ is the tree width of $M(G)$.

We can express their notion of a tree-decomposition of a matroid in the language of $S_k$-trees in the following way. Given any $X \subset E$ the \emph{connectivity} of $X$ is given by
\[
\lambda(X) := r(X) + r(E \setminus X) - r(M),
\]
where $r$ is the rank function of $M$.

If we consider the universe of separations $\ra{S}$ given by the bipartitions of $E$, that is, pairs of the form $(X,E\setminus X)$, it follows that $|X,E\setminus X|_r = \lambda(X)$ is an order function on $\ra{S}$.

Let us define, as before
\[
\cc{F}_k = \{ \sigma \in \N^{{\ra{S}\!}_k} \, : \ \sigma \text{ a star with }\langle \sigma \rangle < k \}.
\]
Diestel and Oum [Lemma 8.4 \cite{DO142}] showed that $M$ has tree-width $<k$ in the sense of Hlin\v{e}n\'{y} and Whittle if and only if there is an $S_k$-tree over $\cc{F}_k$. Explicitly given an $S_k$-tree $(T,\alpha)$ there is a natural map $\tau : E \rightarrow T$ where $\tau(e) = v$ if and only if $e \in B$ for all $(A,B) \in \alpha({\ra{F}\!}_v)$. Conversely given a tree-decomposition  $(T,\tau)$ and an edge $\ra{f} = (t_1,t_2)$ consider the two subtrees $T_1$ and $T_2$ consisting of the component of $T - t_2$ containing $t_1$ and the component of $T-t_1$ containing $t_2$ respectively. We can then define $\alpha(\ra{f}) = \left(\tau^{-1}\left(V(T_1)\right), \tau^{-1}\left( V(T_2)\right)\right)$. In this way we get an equivalence between $S_k$-trees and matroid tree-decompositions, and it is easy to check that the width of a bag $\langle \tau^{-1}(v) \rangle = \langle \sigma_v \rangle_r$.

We first note that Azzuto \cite{A11} showed that Theorem \ref{t:branch} implies the following:
\begin{theorem}
Every matroid has a linked tree-decomposition of width at most $2$tw$(M)$.
\end{theorem}
It is a simple corollary of Theorem \ref{t:link} that this bound can be improved to the best possible.
\begin{cor}
Every matroid has a linked tree-decomposition of width at most tw$(M)$.
\end{cor}

However we can also apply Theorem $\ref{t:lean}$ to give us a generalization of Theorem \ref{t:thomlink} to matroids. If we wish to express this in the framework of Hlin\v{e}n\'{y} and Whittle we could make the following definition. Given disjoint subsets $Z_1, Z_2 \subseteq E$ let us write 
\[
\lambda(Z_1,Z_2):= \min \{ \lambda(X) \,:\, Z_1 \subseteq X, Z_2 \subseteq E \setminus X \}.
\]

\begin{defn}\label{d:leanmat}
A matroid tree decomposition $(T,\tau)$ is called \emph{lean} if for all $k \in \bb{N}$, $t,t' \in T$ and subsets $Z_1 \subseteq \tau^{-1}(t)$ and $Z_2 \subseteq \tau^{-1}(t')$ with $r(Z_1)=r(Z_2)=k$, either $\lambda(Z_1,Z_2) \geq k$ or there exists an edge $(s,s') \in tTt'$ such that $\left(\tau^{-1}\left(V(T_1)\right), \tau^{-1}\left( V(T_2)\right)\right)$ is a $<k$ separation.
\end{defn}

It is a simple argument in the vein of Lemmas \ref{l:stable} and \ref{l:leantree} that if an $S_k$-tree over $\cc{F}_k$ is $\cc{F}_k$-lean then the associated matroid tree-decomposition is lean. Then, Theorem $\ref{t:lean}$ applied to $\cc{F}_k$ gives us the following generalization of Theorem \ref{t:thomlink} to matroids, the main new result in this paper.

\begin{theorem}\label{t:matroid}
Every matroid $M$ admits a lean tree-decomposition of width tw$(M)$.
\end{theorem}

We note that a non-negative non-decreasing submodular function $r : 2^V \rightarrow \N$ is, if normalised such that $r(\emptyset) =0$, a polymatroid set function. So, in the broadest generality our results can be interpreted in terms of tree-decompositions of polymatroidal set functions. However, we are not aware of any references in the literature to such tree-decompositions.

\section*{Acknowledgement}
The author would like to thank Nathan Bowler for useful contributions to the proof of Theorem \ref{t:lean}.

\bibliographystyle{plain}
\bibliography{linked}
\end{document}